\documentclass[11pt]{article}
\usepackage{amsmath,amssymb}
\usepackage{graphicx}
\usepackage{mathrsfs}
\usepackage{bbm}
\usepackage{xcolor}
\usepackage{stmaryrd}

\newtheorem{remark}{Remark}
\newtheorem{theorem}{Theorem}
\newtheorem{lemma}{Lemma}

\newtheorem{proposition}{Proposition}

\newtheorem{corollary}{Corollary}

\def \R{{\mathbb{R}}}
\def \N{{\mathbb{N}}}

\title{Onsager type conjecture and renormalized solutions for the relativistic Vlasov--Maxwell system}
\author{Claude Bardos  \footnotemark[1] \and Nicolas Besse \footnotemark[2] \and Toan T. Nguyen \footnotemark[3]} 

\begin{document}

\maketitle

\begin{center}
\emph{This paper is dedicated to Walter Strauss
\\on the occasion of his 80th birthday, as token of friendship and admiration 
\\in particular for his contribution to the mathematical theory of Vlasov-Maxwell systems.
}
\end{center}

\renewcommand{\thefootnote}{\fnsymbol{footnote}}

\footnotetext[1]{
Laboratoire J.-L. Lions,
Universit\'e Pierre et Marie Curie - Paris 6  
BP 187, 4 place Jussieu,
75252 Paris, Cedex 5, France
({\tt claude.bardos@gmail.com})}

\footnotetext[2]{
Laboratoire J.-L. Lagrange, 
UMR CNRS/OCA/UCA 7293, 
Universit{\'e} C\^ote d'Azur,
Observatoire de la C\^ote d'Azur,
Bd de l'observatoire CS 34229,
06300 Nice, Cedex 4, France.
({\tt Nicolas.Besse@oca.eu})}

\footnotetext[3]{
Department of Mathematics, 
Penn State University, State College,
PA, 16803, USA
({\tt nguyen@math.psu.edu})}

\begin{abstract}
  In this paper we give a proof of an Onsager type conjecture on conservation of energy and entropies of weak solutions to the relativistic
  Vlasov--Maxwell equations. As concerns the regularity of weak solutions, say in 
  Sobolev spaces $W^{\alpha,p}$, we determine Onsager type exponents $\alpha$ that
  guarantee the conservation of all entropies. In particular,
  the Onsager exponent $\alpha$ is smaller than $\alpha = 1/3$ established for fluid models. Entropies conservation is equivalent
  to the renormalization property, which have been introduced by DiPerna--Lions for studying
  well-posedness of passive
  transport equations and collisionless kinetic equations. For smooth solutions
  renormalization property or entropies conservation are simply the consequence of the chain rule.
  For weak solutions the use of the chain rule is not always justified. Then arises the question about
  the minimal regularity needed for weak solutions to guarantee such properties. 
  In the DiPerna--Lions and Bouchut--Ambrosio theories, renormalization property holds under sufficient
  conditions in terms of the regularity of the advection field, which are roughly speaking an entire
  derivative in some Lebesgue spaces (DiPerna--Lions) or an entire derivative
  in the space of measures with finite total variation (Bouchut--Ambrosio).
  In return there is no smoothness requirement for the advected density, except some natural a priori bounds.
  Here we show that the renormalization property holds for  an electromagnetic field with only a fractional
  space derivative in some Lebesgue spaces. To compensate this loss of derivative for the electromagnetic field,
  the distribution function requires an additional smoothness, typically fractional Sobolev differentiability
  in phase-space.
  As concerns the conservation of total energy, if the macroscopic kinetic
  energy is in $L^2$, then total energy is preserved.

~

\end{abstract}

\newpage

\section{Introduction}
The dimensionless relativistic Vlasov--Maxwell system reads,
\begin{equation}
\label{eqn:RV}
\partial_t f + v\cdot\nabla_xf + (E+ v\times B)\cdot \nabla_{\xi}f =0,
\end{equation}
\begin{equation}
\label{eqn:M1}
\partial_tE-\nabla \times B=-j, \quad \partial_t B + \nabla \times E =0,  
\end{equation}
\begin{equation}
\label{eqn:M2}
\nabla \cdot E=\rho, \quad \nabla \cdot B =0,
\end{equation}
where $t\in\R$, $x\in \R^3$, ${\xi}\in\R^3$, and $v={\xi}/\sqrt{1+|{\xi}|^2}$ represent time, position, momentum
and velocity of particles, respectively. The distribution function of particles $f=f(t,x,{\xi})$
satisfies the Vlasov equation \eqref{eqn:RV} with acceleration given by the Lorentz force $F_L=E+ v\times B$,
while the electromagnetic field $E=E(t,x)$ and $B=B(t,x)$ satisfies
Maxwell's equations \eqref{eqn:M1}-\eqref{eqn:M2}. The coupling between the Vlasov equation and Maxwell's equations
occurs through the source terms of Maxwell's equations, which are the charge density $\rho=\rho(t,x)$
and the current density $j=j(t,x)$.
These densities are defined as the first ${\xi}$-moments 
of the phase-space density of particles $f$, namely,
\begin{equation}
\label{eqn:rhoj}
\rho(t,x)=\int_{\R^3}f(t,x,{\xi})\,d{\xi}, \quad j(t,x)=\int_{\R^3} vf(t,x,{\xi})\,d{\xi}.  
\end{equation}
The initial value problem associated to the system \eqref{eqn:RV}-\eqref{eqn:rhoj} requires initial conditions
given by,
\begin{eqnarray}
&&f(0,x,{\xi})=f_0(x,{\xi})\geq0, \label{eqn:IC:f0} \\
&&E(0,x)=E_0(x), \ \ B(0,x)=B_0(x), \ \ \nabla\cdot E_0 = \rho_0=\int_{\R^3}f_0\,d{\xi}, \ \ \nabla\cdot B_0=0. \label{eqn:IC:EB0}  
\end{eqnarray}
In addition for the well-posedness of Maxwell's equations \eqref{eqn:M1}-\eqref{eqn:M2}, the densities
of charge $\rho$ and current $j$ must satisfy a compatibility condition given by the charge conservation law,
\begin{equation}
\label{eqn:4C}
\partial_t \rho +\nabla\cdot j =0. 
\end{equation}
This continuity equation is automatically satisfied if the Vlasov equation \eqref{eqn:RV} is satisfied
since it can be recovered by integration in momentum variable of the Vlasov equation. Let us note that 
Maxwell--Gauss equations \eqref{eqn:M2} are satisfied at any time if they are satisfied initially. Indeed, it is a consequence
of time integration of the divergence of the Maxwell--Faraday--Amp\`ere equations \eqref{eqn:M1},
in combination with the continuity equation \eqref{eqn:4C} and initial conditions \eqref{eqn:IC:EB0}.

The Vlasov equation \eqref{eqn:RV} has, at least formally, infinitely many invariants. Indeed,
let $\mathcal{H}:\R \rightarrow \R$ be any smooth function. Multiplying \eqref{eqn:RV}
with $\mathcal{H}'(f)$ and applying the chain rule, we then obtain, 
\begin{equation}
  \label{FRRVM}
  \partial_t \mathcal{H}(f) + v\cdot\nabla_x\mathcal{H}(f) + (E+ v\times B)\cdot \nabla_{\xi}\mathcal{H}(f) =0.
\end{equation}  
A solution $f$ to \eqref{eqn:RV} in the sense of distributions is said to be a renormalized solution if
for any smooth nonlinear function  $\mathcal{H}$, $f$ also solves \eqref{FRRVM} in the sense of distributions.
We say that the field $(v,F_L)$ satisfies the renormalization property if any solution $f$ to \eqref{eqn:RV}
in the sense of distributions is a renormalized solution. The renormalization technique appeared
in the well-posedness of passive advection equations and ODEs \cite{DL89b},
in the analysis of the Boltzmann equation \cite{DL89c}, in the theory of weak solutions
of the compressible Navier-Stokes equations \cite{Lio98} and in the theory of weak solutions
of collisionless kinetic equations such as the Vlasov--Poisson system \cite{DL88a, DL88b}. The
groundbreaking work \cite{DL89b} has highlighted the fundamental link between renormalized solutions
to the passive transport equation,
\begin{equation}
  \label{PDE_L}
\partial_t u + b\cdot\nabla u =0, \quad u:[0,T]\times\R^d\rightarrow \R, \ \ b:[0,T]\times\R^d \rightarrow \R^d, 
\end{equation}  
and the well-posedness theory for the associated ODE,
\begin{equation}
  \label{ODE_L}
  \partial_t X(t,x) = b(t,X(t,x)), \ \ t\in[0,T], \ \ X(0,x)=x\in\R^d, 
\end{equation}  
where $b$ is a non-smooth vector field. Similarly
to entropy conditions for hyperbolic conservations laws,
renormalization property provides additional stability under weak convergence. Indeed
renormalized solutions come with a comparison principle, which allows to show
uniqueness of renormalized solutions and some stability results for sequences of
solutions. In return uniqueness at the PDE level \eqref{PDE_L} implies uniqueness at the
ODE level \eqref{ODE_L}. It was first show in \cite{DL89b} that the renormalization property holds provided
$b\in L_t^1W_x^{1,p}$ with $p\geq  1$, plus a bounded divergence and a global space growth estimate on $b$
(see also \cite{LL04} for the case $W_{\rm loc}^{1,1}$). Moreover,
there is no additional regularity assumption for $u$ except its boundedness or
some $L^p$-bounds. This result was extended to $b\in L_t^1BV_x$ with $\nabla\cdot b \in L_{t,x}^1$,
first in \cite{Bou01} for the Vlasov equation (see also \cite{Hau04} for a related result),
and then in \cite{Amb04} for the general case (see also \cite{CL02}).
Very recently, in \cite{ACF15} the authors develop a local version of the DiPerna--Lions' theory under
no global assumptions on the growth estimate of $b$. We refer the reader to \cite{Amb17} for a recent survey.

For the Vlasov--Poisson system when $f$ is merely $L^1$, the product $Ef$ does not belong
to $L_{\rm loc}^1$. Therefore higher integrability assumptions on $f$ are needed to give a meaning
to the Vlasov--Poisson equation in the sense of distributions. For example, when $d=3$, for the term $Ef$ to belong to $L_{\rm loc}^1$
one needs to have $f\in L^p$ with $p=(12+3\sqrt{5})/11$ (see for instance \cite{DL88a, DL88b}).
To drop out this higher integrability hypotheses, in \cite{DL88a, DL88b}
the authors considered the concept of renormalized solutions
and obtained global existence provided that the total energy is finite and $f_0\log(1+f_0)\in L^1$.
In addition, under some suitable integrability hypotheses on $f$, they can show that the concepts
of weak and renormalized solutions are equivalent. For bounded density $f$,
renormalization property holds because elliptic regularity of the Poisson equation
leads to $E\in W^{1,p}$, with $p>1$ (see \cite{DL88a, DL88b}).
For the Vlasov--Maxwell system the only available global existence result is in \cite{DL89a}, where
the authors have constructed  weak solutions for which it is not possible to show the renormalization
property. Indeed, the best electromagnetic-field regularity, obtained so far 
for the DiPerna--Lions weak solutions, is in \cite{BB18}, where the authors show that
the electromagnetic field $(E,B)$ belongs to $H_{\rm loc}^s(\R_\ast^+\times\R^3)$, with $s=6/(13+\sqrt{142})$,
if the macroscopic kinetic  energy is in $L^2$.

Regularity of rough vector field considered above, i.e 
Sobolev or BV vector fields, is somehow like the Lipschitz case because
there is always a control (in Lebesgue spaces or in the space of measures with finite total variation)
on an entire derivative of the vector field. By contrast, when $b$
is not Lipschitz-like, the use of the chain rule is no longer justified,
and many counterexamples to renormalization have been obtained in 
\cite{Aiz78,Dep03,CLR03,ABC13,ABC14,ACM14,ACM18,CGSW15,CGSW17,YZ17}.

Here, we show that the renormalization property holds for an electromagnetic field with  only a fractional
derivative in some Lebesgue spaces, i.e. $E, B \in L_t^\infty W_x^{\beta,q}$, with $0<\beta <1$ and
$1\leq q \leq \infty$. To compensate this loss of derivative for the electromagnetic field,
the density $f$ requires additional smoothness, typically fractional Sobolev differentiability
in phase-space, i.e. $f \in L_t^1W_{x,\xi}^{\alpha,p}$,
with $0<\alpha <1$ and $1\leq p \leq \infty$. We determine Onsager type exponents \cite{Eyi18} $\alpha$ and $\beta$, which
ensure conservation of all entropies and guarantee that the renormalization property holds.
As concerns the conservation of total energy, if the macroscopic kinetic
energy is in $L^2$, we then show that total energy is preserved.
A comparable work has been done in \cite{AW18} for the renormalization of an active scalar transport equation.

A similar situation occurs with systems of conservation laws of continuum physics, which
are  endowed with natural companion laws: the so called the entropy conditions (inequality versus equality)
coming from the second law of thermodynamics.
In \cite{GMSG18, BGSGTW18} the authors have determined the critical regularity of weak solutions
to a general system of conservation laws to satisfy an associated entropy
conservation law as an equality. They obtained the famous Onsager exponent $1/3$ \cite{Ons49}.
The first result of this kind was obtained in \cite{CET94} (see also \cite{Eyi94}), where
the authors have shown that weak solutions of the incompressible Euler equations conserve energy provided
they possess fractional Besov differentiability of order greater than $1/3$.
Such result has been extended in various directions: In  \cite{DR00,CCFS08,FW18}
the Besov criterium has been optimized; in \cite{LS16,FGSGW17,Yu17,DE18, GMSG18} the authors have considered
compressible Euler, Navier-Stokes and magnetohydrodynamic equations;
works \cite{RRS17,RRS18,BT18,BTW18,DN18} include boundary effects.

\section{Basic properties}
In this section, we recall the basic properties of the relativistic Vlasov--Maxwell system,
which are valid for any smooth solution $(f,E,B)$, vanishing at infinity. Theses formal properties,
in particular natural a priori estimates, are the key cornerstones for proving
the local-in-time well-posedness of this system \cite{Gla96}. 
Consider the following set of equations, 
\begin{equation}
\label{eqn:RV:2}
\partial_t f + v\cdot\nabla_xf + (E+ v\times B)\cdot \nabla_{\xi}f =0, \quad (t,x,\xi)\in \R^+\times \R^3 \times \R^3,
\end{equation}
\begin{equation}
  \label{eqn:M1:2}
  \partial_t E-\nabla \times B=-j, \quad \partial_t B + \nabla \times E =0,  
\end{equation}
\begin{equation}
  \label{eqn:rhoj:2}
\rho(t,x)=\int_{\R^3}f(t,x,{\xi})\,d{\xi}, \quad j(t,x)=\int_{\R^3} vf(t,x,{\xi})\,d{\xi},  
\end{equation}
\begin{equation}
  \label{def:gamma}
  \gamma = \sqrt{1+|{\xi}|^2}, \quad v= \nabla_\xi \gamma = \frac{\xi}{\sqrt{1+|{\xi}|^2}}.
\end{equation}
Observe that once the current density $j$ and initial data $(E_0,B_0)$ are given, the Maxwell equations
\eqref{eqn:M1:2} are well defined. Indeed the Maxwell operator  $M$ defined by,
\begin{equation}
  \label{def:MO}
  X\mapsto MX= \left(
  \begin{array}{r}
    -\nabla \times B \\
    \nabla \times  E
  \end{array}\right), \quad \mbox{ with } \quad
   X= \left(
  \begin{array}{r}
    E \\
    B
  \end{array}\right),
\end{equation}  
is the generator of a strongly continuous unitary group $t\mapsto S(t):=\exp(-iMt)$ in $L^2(\R^3)$ \cite{Yos65, EN99, ABKN11}.
If $(E(t=0),B(t=0))=(E_0,B_0) \in L^2(\R^3)$ and $ j\in L^1(\R^+; L^2(\R^3)) $, then,
using the properties of the group $S(t)$ and the Duhamel formula, we can show that the 
the solution $(E, B)$ to \eqref{eqn:M1:2} belongs to $\mathscr{C}(\R^+; L^2(\R^3))$.
Moreover for any $s\ge 0$, the $H^s$ regularity is preserved,
i.e. the previous statement remains valid if we replace $L^2(\R^3)$ by $H^s(\R^3)$.
In the same way, once the  smooth electromagnetic field $(E,B)$ and  initial
data $f_0(x,\xi)$ are given, the Vlasov equation is then  well defined. Indeed,
introducing the characteristics curves $t\mapsto (X(t),\Xi(t))$, which are the unique and smooth solution
to the ODEs,
\begin{eqnarray}
  \label{eqn:LFlow}
  &&\frac{dX}{dt}(t)=v(\Xi(t)), \quad  \frac{d\Xi}{dt}(t)=E(t,X(t)) + v(\Xi(t))\times B(t,X(t)), \\
  && X(0; 0,x,\xi)=x, \quad \Xi(0;0,x,\xi)=\xi,  \label{eqn:LFCI}
\end{eqnarray}  
the Lagrangian solution to \eqref{eqn:RV:2} is given by (e.g., see \cite{BGP00})
\begin{equation}
  \label{eqn:LSV}
  f(t,x,\xi) = f_0(X(0;t,x,\xi),\Xi(0,t,x,\xi)).
\end{equation}  

The relativistic Vlasov--Maxwell system \eqref{eqn:RV:2}-\eqref{def:gamma} satisfies some formal
conservation laws, summarized in 

\begin{proposition}
\label{prop:FPVM}
  Let  $(f,E,B)$ be a smooth solution, vanishing at infinity, to the relativistic Vlasov--Maxwell system \eqref{eqn:RV:2}-\eqref{def:gamma}. Then the following a priori estimates hold:
  \begin{itemize}
  \item[1.] (Maximum principle).  $\ 0\leq m \leq f_0 \leq M <\infty\ $ implies  $ \  m \leq f(t) \leq M$ for all $t>0$.
  \item[2.] ($L^p$-norm conservation). For all $t\geq 0$,  and $\ 1\leq p \leq \infty$, one has, 
    $\ \|f(t) \|_{L^p(\R^6)}=\|f_0 \|_{L^p(\R^6)}$.
  \item[3.] (Entropies). For any function $\mathcal{H}\in\mathscr{C}^1(\R^+; \R^+)$, one has for all $t\geq 0$,
    \[
     \frac{d}{dt}\int_{\R^3} \int_{\R^3} \mathcal{H}(f(t))\, d\xi dx =0.
    \]
  \item[4.] (Energy conservation).  For all $t\geq 0$ one has,
    \[
    \frac{d}{dt} \left( \int_{\R^3} \int_{\R^3} (\gamma(\xi)-1) f(t) \,d\xi dx + \frac12 \int_{\R^3} (|E(t)|^2 + |B(t)|^2)\, dx \right)=0. 
    \]
  \item[5.] (Momentum conservation).    For all $t\geq 0$ one has, 
    \[
    \frac{d}{dt} \left( \int_{\R^3} \int_{\R^3} \xi f(t)\, d\xi dx + \int_{\R^3} (E(t)\times B(t))\, dx \right)=0. 
    \]
    \end{itemize}
\end{proposition}  

\begin{proof}
The proof is standard and  can be found, for instance, in \cite{BGP00}.
\end{proof}

\begin{remark}
  \label{rmk:OFP}
  Properties of Proposition~\ref{prop:FPVM} are key ingredients to obtain the global-in-time
  existence of weak solutions \cite{DL89a, Gla96, Rei04} and the local-in-time existence, uniqueness and stability
  of classical solutions (e.g. see \cite{Gla96} and references therein). Properties of Proposition~\ref{prop:FPVM}
  are also independent of other a priori invariances described below. Indeed from Maxwell--Faraday equation, $\partial_t B +\nabla \times E=0$,
  we deduce that $\partial_t\nabla \cdot B=0$, which leads to $\nabla \cdot B(t)=0$ for all $t>0$, if initially  $\nabla \cdot B_0=0$.
  In a similar way, from Maxwell--Amp\`ere equation, $\partial_t E - \nabla \times B=-j$, we deduce that $\partial_t (\nabla \cdot E) + \nabla \cdot j=0$.
  Using the charge conservation law \eqref{eqn:4C} (obtained by integration of the Vlasov equation \eqref{eqn:RV:2} with respect to $\xi$)
  we then obtain $\partial_t (\rho-\nabla \cdot E)=0$, which leads to $\nabla \cdot E(t)=\rho(t)$ for all $t>0$, if initially
  $\nabla \cdot E_0 =\rho_0$.
\end{remark}
  
\section{Renormalization property and entropies conservation}

\subsection{Notation}
We denote by $\R^+$ the non-negative real numbers, by
$\mathcal{D}(\R^d)$ the space of indefinitely differentiable with compact support, and
by $\mathcal{D}'(\R^d)$ the space of distributions. We also denote by
$\mathcal{S}(\R^d)$, the space of indefinitely differentiable and rapidly decreasing functions, and 
$\mathcal{S}'(\R^d)$ the dual of $\mathcal{S}(\R^d)$, i.e. the space of tempered distributions. We use the notation $B_{p,q}^\alpha$ (
$0<\alpha <1$, $1\leq p \leq \infty$, $1\leq q \leq \infty$) for Besov spaces, the
definition of which, can be found e.g., in \cite{Ada75, BL76, Tri83, Tri01}.
The notation $W^{\alpha,p}$ ($0<\alpha <1$, $1 \leq p \leq \infty$) stands for the generalized Sobolev spaces of fractional order,
whose precise definition can also be found e.g., in \cite{Ada75, BL76, Tri83, Tri01}.
Let us simply recall first $W^{\alpha,p}(\R^d)=B_{p,p}^{\alpha}(\R^d)$ for $\alpha$ positive but not an integer and $1\leq p \leq\infty$,
and secondly the continuous embeddings: $B_{p,1}^{\alpha}(\R^d) \subset W^{\alpha,p}(\R^d)\subset B_{p,\infty}^{\alpha}(\R^d)$,
with $1\leq p\leq \infty$. We also define the functional space $L_\gamma^1$ such that,
\begin{equation}
  \label{def:L1gamma}
L_\gamma^1= \left\{ f \geq 0 \ \ \mbox{a.e} \ \ \left | \right. \ \ \|f\|_{L_\gamma^1(\R^6)}:=\int_{\R^6}   \gamma f \, dxd{\xi}  < +\infty\right\}.
\end{equation}
Moreover we define the function space $\mathscr{E}$ such that,
\begin{equation}
  \label{entropy_space}
  \mathscr{E} = \left\{\mathcal{H}:\R^+ \mapsto\R^+; \  \mathcal{H}  \mbox{ is non-decreasing}, \ \ \mathcal{H}\in \mathscr{C}^1(\R^+;\R^+),
  \ \  \lim_{\sigma \rightarrow +\infty }\frac{\mathcal{H}(\sigma)}{\sigma} =+\infty \right\}.
\end{equation}

\subsection{Main theorems}
In this section we present our main results. For this, we need to recall the DiPerna--Lions
theorem, which is the only existing result concerning the existence of global-in-time (weak) solutions 
to the Vlasov--Maxwell system in $\R^6$.

\begin{theorem}{(DiPerna--Lions \cite{DL89a}).}
  \label{thm:DL89}
  Let $f_0\in L_\gamma^1\cap L^\infty(\R^6)$, and $E_0,\, B_0 \in L^2(\R^3)$, be initial conditions 
  which satisfy the constraints,
  \begin{equation*}
    \label{thm:eq:IC1}
    \nabla\cdot B_0=0, \quad \nabla\cdot E_0 = \rho_0=\int_{\R^3}f_0\,d{\xi}, \quad \mbox{in} \ \ \mathcal{D}'(\R^3).
  \end{equation*}
  Then, there exists a global-in-time weak solution of the relativistic Vlasov--Maxwell system, i.e. there exists
  functions,
  \begin{equation}
    \label{reg_prop_wsrvm}
    f\in L^\infty\big(\R^+;L_\gamma^1\cap L^\infty(\R^6)\big),\quad
    E,\, B \in  L^\infty\big(\R^+;L^2(\R^3)\big),  \quad
      \rho,\, j \in  L^\infty\big(\R^+;L^{4/3}(\R^3)\big),
  \end{equation}
  such that $(f,E,B)$ satisfy  \eqref{eqn:RV:2}-\eqref{eqn:M1:2} in the sense of distributions, with
  $\rho,\, j$ defined in terms of \eqref{eqn:rhoj:2}. Constraints equations \eqref{eqn:M2} and the charge conservation law
  \eqref{eqn:4C} are statisfied in the sense of distributions.

  In addition, the mapping $t\mapsto f(t)$ (resp. $t\mapsto (E(t),B(t))$) is continuous with respect to the following topologies: the
  standard topology in the space of distributions $\mathcal{D}'(\R^6)$ (resp.  $\mathcal{D}'(\R^3)$),
  the weak topology of $L^2(\R^6)$ (resp. $L^2(\R^3)$), and the strong topology of $H^{-s}(\Omega)$ for any $s>0$ and any
  bounded subset $\Omega$ of $\R^6$ (resp. $\R^3$).

  Futhermore, the total mass, 
  \begin{equation*}
    \int_{\R^3}\int_{\R^3} f(t)dxd\xi,
  \end{equation*} 
  is independent of time, and one has,
  \begin{equation*}
     \|f(t) \|_{L^p(\R^6)} \leq \|f_0 \|_{L^p(\R^6)} \ \mbox{ a.e. } \ 
      t\geq 0,  \ \mbox{ for } \ 1\leq p \leq +\infty,
    \ \ \  \mbox{ and  } \ \ \  \mathcal{E}(t)\leq \mathcal{E}_0<\infty\ \mbox{ a.e. }\ t\geq 0,  
  \end{equation*}
  with the definition,
  \begin{equation}
    \label{thm:eq:IC2}
     \mathcal{E}(t):= \int_{\R^6} \gamma f(t) \, d{\xi}dx+ \frac12 \int_{\R^3} (|E(t)|^2 + |B(t)|^2)\,dx.  
  \end{equation}
 \end{theorem}  

\begin{remark}
  \begin{itemize}
  \item[1.] In \cite{Rei04}, the author shows that  solutions of Theorem~\ref{thm:DL89} preserve
    all $L^p$-norms and the mass, i.e. for all $t>0$,
    \[
    \int_{\R^3}\int_{\R^3} f(t)dxd\xi =  \int_{\R^3}\int_{\R^3} f_0dxd\xi, \ \ \  \mbox{ and } \ \ \ 
    \|f(t) \|_{L^p(\R^6)} = \|f_0 \|_{L^p(\R^6)}, \quad  1\leq p \leq +\infty.
    \]
  \item[2.] Using lower semi-continuity, weak solutions of Theorem~\ref{thm:DL89} satisfy, for
    all $\mathcal{H}\in \mathscr{C}^1(\R^+; \R^+)$,
    \[
      \int_{\R^6} \mathcal{H}(f(t)) \,d{\xi} dx  \leq 
    \int_{\R^6} \mathcal{H}(f_0) \, d{\xi} dx ,\quad  \mbox{ for } \  t\geq 0.
    \]
  \end{itemize}
\end{remark}  

Now we intend to produce supplementary sufficient regularity conditions, which will imply the validity
of supplementary conservation laws. As the first step this is the aim of Theorem~\ref{thm:entropies} below:
indeed we first  give sufficient regularity hypotheses which couple
the regularity of the distribution function $f$ with the regularity of the electromagnetic field $(E,B)$. In the second step,
we use Theorem~\ref{thm:entropies} and the results of \cite{BB18} on
the regularity of  DiPerna--Lions  weak solutions, to obtain Corollary~\ref{cor:entropies} below,
which involves only sufficient regularity condition on the distribution function $f$.
As concerns the renormalization property and  entropies conservation, we have,
\begin{theorem}
  \label{thm:entropies}
  Let $(f,E,B)$ be a weak solution of the relativistic Vlasov--Maxwell system \eqref{eqn:RV:2}-\eqref{def:gamma}, given
  by Theorem~\ref{thm:DL89}. Assume that with,  
  \begin{equation}
    \label{reg-feb-entropy-1}
    \alpha, \beta \in \R,\quad \ \  0<\alpha,\, \beta<1,  \quad  \mbox{ and }\quad \alpha \beta + \beta + 3\alpha -1 >0,    
  \end{equation}
  this weak solution satisfies for some $(p,q) \in \N_\ast^2 $ with, 
  \begin{equation}
    \begin{aligned} 
      &\frac{1}{p} + \frac{1}{q}=\frac{1}{r} \le 1 \quad \hbox{if} 
      \quad  1\le p,q<\infty,\\
      &1\le  r < \infty \hbox{ is arbitrary if }  p=q=\infty,  \label{reg-feb-entropy-2}
    \end{aligned}
  \end{equation}
  the supplementary regularity hypotheses,
  \begin{equation}
    \label{reg-feb-entropy-3} 
      f\in L^\infty\big(0,T; W^{\alpha,p}(\R^6)\big), \quad \mbox{ and } \quad  E,B \in L^\infty\big(0,T; W^{\beta,q}(\R^3)\big).
  \end{equation}
  Then for any entropy function $\mathcal{H}\in \mathscr{C}^1(\R^+;\R^+)$, we have
  the renormalization property,
  \begin{equation}
    \label{entropeq_R}
    \partial_t (\mathcal{H}(f)) + \nabla_x\cdot (v\mathcal{H}(f)) + 
    \nabla_{\xi}\cdot (F_L\mathcal{H}(f)) = 0, \quad \mbox{in } \ \mathcal{D}'((0,T)\times\R^6).
  \end{equation}
  Moreover, if $\mathcal{H}\in \mathscr{E}$ and the map,  
  \begin{equation}
    \label{H:UI}
    t\mapsto f(t,\cdot,\cdot) \mbox{ is uniformly integrable in } \R^6, \mbox{ for  a.e. } t\in [0;T],
  \end{equation}
  then we have the local entropy conservation laws,
  \begin{equation}
    \label{entropeq_L1}
    \partial_t \left(\int_{\R^3}d{\xi}\,\mathcal{H}(f)\right) + \nabla_x\cdot \left( \int_{\R^3}d{\xi}\, v\mathcal{H}(f)\right) 
    = 0, \quad \mbox{in } \ \mathcal{D}'((0,T)\times\R^3),
  \end{equation}
  \begin{equation}
    \label{entropeq_L2}
    \partial_t \left(\int_{\R^3}d{x}\,\mathcal{H}(f)\right) + \nabla_{\xi}\cdot \left( \int_{\R^3}d{x}\, F_L\mathcal{H}(f)\right) 
    = 0, \quad \mbox{in } \ \mathcal{D}'((0,T)\times\R^3),
  \end{equation}
  and the global entropy conservation law,
  \begin{equation}
    \label{entropeq_G}
    \int_{\R^6} \mathcal{H}(f(t,x,{\xi})) \, d{\xi}dx =
    \int_{\R^6} \mathcal{H}(f(s,x,{\xi})) \, d{\xi}dx,\quad  \mbox{ for } \  0<s\leq t<T.
  \end{equation}
\end{theorem}

The proof of  Theorem~\ref{thm:entropies} is postponed to Section~\ref{ssec:Proof-Thm-entropy}.
A few remarks are now in order.

\begin{remark}
  \label{VP_VM_cases}
  In fact,  Theorem~\ref{thm:entropies} is also true for the Vlasov--Poisson and the non-relativistic
  Vlasov--Maxwell systems, under the same regularity assumptions.
\end{remark}

\begin{remark}
  \label{restrictions_property}
  \begin{itemize}

\item[1.] In fact, Theorem~\ref{thm:entropies} still holds when we replace Sobolev spaces $W^{\alpha,p}$ (resp.
  $W^{\beta,q}$) by Besov spaces  $B_{p,\infty}^{\alpha+ \epsilon}$ (resp.  $B_{q,\infty}^\beta$), with $\epsilon>0$.
  Indeed, even if  Besov  spaces $B_{p,\infty}^\alpha$ do not share the restriction property (needed for
  proving commutator estimates of Lemma~\ref{lem:commutators}), we still have
  the following result (see \cite{Jaf95, AMS13, Bra18}): let $N\geq 2$, $\ 1\leq d < N$, $\ 0<p<q\leq \infty$,
  $\ \alpha'>N(1/p-1)_+$,
  and $\ f\in B_{p,\infty}^{\alpha'}(\R^N)$. Then,
  \[
  f(\cdot, y) \in  \bigcap_{\alpha < \alpha'}  B_{p,\infty}^{\alpha}(\R^d), \quad \mbox{for a.e } \ y\in\R^{N-d}.
  \]
  Therefore, in the Besov-spaces framework, replacing $\alpha$ by $\alpha +\epsilon$ with $\epsilon>0$
  in \eqref{reg-feb-entropy-1}, we observe that
  the condition  $\alpha \beta + \beta + 3\alpha -1 >0$ keeps the same, whereas the phase-space regularity of $f$ is
  slightly better than $B_{p,\infty}^{\alpha}$. Since the interpolation between $B_{p,\infty}^{\alpha+\epsilon}$ and
  $B_{p,p}^{\alpha}$ is $B_{p,r}^{\alpha'}$, with $\alpha < \alpha' < \alpha+\epsilon$, and $1\leq r \leq \infty$, (e.g.,
  Theorem~6.4.5 in \cite{BL76}),  we then have $B_{p,\infty}^{\alpha+\epsilon} \subset W^{\alpha,p}$.

\item[2.] Theorem~\ref{thm:entropies} also includes the H\"older spaces where,
  \begin{equation*}
    \label{reg-feb-entropy-3-holder} 
    f\in L^\infty\big(0,T; \mathscr{C}^{0,\alpha}(\R^6)\big), \quad \mbox{ and }
    \quad  E,B \in L^\infty\big(0,T; \mathscr{C}^{0,\beta}(\R^3)\big).
  \end{equation*}
  It corresponds to case where $p=q=\infty$ in \eqref{reg-feb-entropy-3-holder}, since
  $\mathscr{C}^{0,\alpha}= W^{\alpha,\infty}$. 
  
  \end{itemize}
\end{remark}

\begin{remark}
  \label{Eyink18}
  Our result is almost in agreement with the structure-function scaling exponents
  derived in the study of dissipative anomalies in nearly collisionless plasma turbulence \cite{Eyi18}.
  \begin{itemize} 
  \item [1.] Here, the rigorous analysis is purely deterministic and regularity conditions
  \eqref{reg-feb-entropy-1}-\eqref{reg-feb-entropy-3}
  give a sufficient condition for the conservation of entropies for any individual solution as in \cite{Eyi18}.
  In other words, by contraposition, a necessary condition for anomalous dissipation/non-conservation of entropies is, 
  $\alpha \beta + \beta + 3\alpha -1 <0$ with $0<\alpha, \beta <1$. Nevertheless, this condition is not sufficient.
  Indeed, as in fluid mechanics with the Onsager critical regularity exponent $1/3$ \cite{BT10, Sze11, BTW12},
  this necessary condition does not rule out the existence of some solutions that are less regular than the
  critical regularity (exponent) and that also satisfy the absence of anomalous entropy dissipation.

\item[2.] In \cite{Eyi18} the author obtains, in a particular case, the critical exponent value $\alpha=\sqrt{5}-2$,
  assuming that $f\in L^\infty(0,T; B_{p,\infty}^{\alpha}(\R_x^3; B_{p,\infty}^{\alpha}(\R_\xi^3)))$ and
  $E,B \in L^\infty(0,T;B_{p,\infty}^{\alpha}(\R^3))$, with $p\geq 3$.
  From Remark~\ref{restrictions_property} on the restriction property of Besov spaces, in order to obtain 
  $f\in L^\infty(0,T; B_{p,\infty}^{\alpha}(\R_x^3; B_{p,\infty}^{\alpha}(\R_\xi^3)))$, we must require the
  distribution function $f$ to belong to the functional space
  $ L^\infty(0,T; B_{p,\infty}^{\alpha+\epsilon}(\R^6))$,  with $\epsilon >0$.
  Now, taking $\alpha=\beta$, the condition $\alpha \beta + \beta + 3\alpha -1 >0$ in \eqref{reg-feb-entropy-1}
  becomes $\alpha^2+4\alpha-1>0$, which is satisfied for $\alpha > \sqrt{5}-2$. We then recover the same
  critical exponent value $\alpha=\sqrt{5}-2$, but for  $f\in L(0,T;W^{p,\alpha}(\R^6))$ and
  $E,B \in L^\infty(0,T;W^{q,\alpha}(\R^3))$, with $1/p+1/q\leq 1$.
  Therefore our regularity conditions  \eqref{reg-feb-entropy-1}-\eqref{reg-feb-entropy-3} are
  weaker, but less restrictive that those of \cite{Eyi18}. Indeed we have  $B_{p,\infty}^{\alpha+\epsilon} \subset W^{\alpha,p}$,
  $\forall \epsilon >0$, and the condition $1/p+1/q\leq 1$ is less restrictive than
  the condition $p=q\geq 3$.

\item[3.] In \cite{Eyi18} the author obtains a refined version of the condition \eqref{reg-feb-entropy-1}, by considering
  anisotropic regularity for the distribution function $f$ between the space of velocities and the physical space,
  namely  $f\in L^\infty(0,T; B_{p,\infty}^{\kappa}(\R_x^3; B_{p,\infty}^{\sigma}(\R_\xi^3)))$.
  From Remark~\ref{restrictions_property} on the restriction property of Besov spaces,  this anisotropic regularity implies
  that  $f\in L^\infty(0,T; B_{p,\infty}^{\alpha+\epsilon}(\R^6))$, with $\alpha:=\max\{\kappa,\sigma\}$
  and $\epsilon>0$.  This regularity condition is still more restrictive than our regularity condition, namely 
  $f\in L^\infty(0,T; W^{\alpha,p}(\R^6))$ with the same index $\alpha$.
  In addition anisotropic regularity in phase space is questionable because of the
  following physical argument. Phase-space turbulence involves typical structures known as vortices that are
  the result of the filamentation and the trapping (or wave-particle synchronization) phenomena. The fact
  that characteristic curves roll up in phase space seems to contradict that
  phase-space regularity is anisotropic between the space of velocities and the physical space.
  On the contrary, this mixing motion must propagate regularity versus singularities from one direction to another.
  By constrast, anisotropic regularity between the electromagnetic field $(E,B)$ and the distribution $f$
  is justified and crucial, because the velocity integration of $f$ can lead to additional regularity in
  the physical space for the moments such as charge and current densities, and hence for
  the electromagnetic field (through Maxwell's equations). This is the essence
  of averaging lemma \cite{DL89a} and the spirit of regularity results obtained for the
  Diperna--Lions  weak solutions \cite{BGP04, BB18}. This anisotropy of regularity is handled
  both here and in \cite{Eyi18}.
  \end{itemize}
\end{remark}

\begin{remark}
  \label{rem:uniqueness} In the non-self-consistent case, i.e. when the Lorentz force $F_L$
  is a given external force, renormalization property \eqref{entropeq_R} implies straightforwardly the uniqueness of
  weak solutions of Theorem~\ref{thm:entropies}, if such solutions exist. Indeed, let
  $f^i$, $i=1,2$, be two solutions of the Vlasov equation \eqref{eqn:RV:2}, with initial
  conditions $f_0^i$, $i=1,2$, and where the electromagnetic field $(E,B)$ is prescribed.
  Such solutions satisfy the regularity properties of  Theorem~\ref{thm:entropies},
  in particular \eqref{reg-feb-entropy-3}. Setting $f=f^1-f^2$, and taking $\mathcal{H}(\cdot)=(\cdot)^2$
  ($\mathcal{H}\in \mathscr{E}$), we obtain from Theorem~\ref{thm:entropies},
  \begin{equation*}
    \label{entropeq_RUN}
    \partial_t (\mathcal{H}(f)) + \nabla_x\cdot (v\mathcal{H}(f)) + 
    \nabla_{\xi}\cdot (F_L\mathcal{H}(f)) = 0, \quad \mbox{in } \ \mathcal{D}'((0,T)\times\R^6),
  \end{equation*}
  and
  \begin{equation*}
    \label{entropeq_GUN}
    \int_{\R^6} \mathcal{H}(f(t)) \,d{\xi}dx =
    \int_{\R^6} \mathcal{H}(f_0) \, d{\xi}dx.
  \end{equation*}  
  Therefore, taking $f_0^1=f_0^2$, i.e $f_0=f^1-f^2=0$, we obtain $f=0$ a.e., i.e. $f_1=f_2$ a.e..
  In a similar way we can show the following comparison principle: $f_0^1\leq f_0^2$ a.e. implies 
  $f^1\leq f^2$ a.e..
  Two open issues remain. The first one is the uniqueness of solutions of  Theorem~\ref{thm:entropies},
  which corresponds to the self-consistent case. Of course the existence of solutions of Theorem~\ref{thm:entropies}
  is also an open big problem. Following the program of \cite{DL89a}, the second one is the existence
  and uniqueness of corresponding Lagrangian solutions, i.e. solutions constructed from
  almost-everywhere-well-defined caracteristics curves, as in the smooth framework \eqref{eqn:LFlow}-\eqref{eqn:LSV}.
\end{remark}

\begin{remark}
\label{rem:BouDom}
Another open issue is the case of bounded domains in space, with specular reflection and/or absorbing
conditions \cite{Guo93}. This is not an easy task since, for such natural boundary conditions, some singularities
could occur at the boundary and propagate inside the domain \cite{Guo94, Guo95}.
\end{remark}

From Theorem~\ref{thm:entropies} and the result of \cite{BB18} on
the regularity of the DiPerna--Lions  weak solutions,  we deduce the following corollary,
which involves hypotheses concerning only the distribution function $f$.
\begin{corollary}
  \label{cor:entropies}
  Let $\beta = 6/(13+\sqrt{142})$, and $\alpha\in\R$ solution to,
  \begin{equation}
    \label{condab}
  \alpha \beta + \beta + 3\alpha -1 >0, \quad \mbox{ and } \quad0<\alpha<1.
  \end{equation}
  Let $(f,E,B)$ be a  weak solution 
  to  the relativistic Vlasov--Maxwell system \eqref{eqn:RV:2}-\eqref{def:gamma}, given by Theorem~\ref{thm:DL89}.
  Assume the additional hypotheses: initial conditions $(E_0, B_0)$ belong to $H^1(\R^3)$,
  the distribution function $f$ satisfies the supplementary integrability condition,
  \begin{equation}
    \label{thm:eq:MKEC0}
    \int_{\R^3}  \gamma f \, d{\xi}\in L^\infty\big(0,T; L^2(\R^3)\big), 
  \end{equation}
  and  the regularity assumption,
  \begin{equation}
    \label{reg-f-entropy}
    f\in L^\infty\big(0,T;  H^{\alpha}(\R^6)\big).  
  \end{equation}
  Then for any entropy function $\mathcal{H}\in \mathscr{C}^1(\R^+;\R^+)$,
  renormalization property \eqref{entropeq_R} holds. Moreover, if $\mathcal{H}\in \mathscr{E}$ and the map
  $t\mapsto f(t,\cdot,\cdot)$  is uniformly integrable in $\R^6$ for  a.e.  $t\in \R^+$,
  then local entropy conservation laws \eqref{entropeq_L1}-\eqref{entropeq_L2}, as well as,
  global entropy conservation law \eqref{entropeq_G} hold.
\end{corollary}

\begin{namedproof} {\it of Corollary}~\ref{cor:entropies}.
  Using  assumptions $E_0, B_0 \in H^1(\R^3)$ and \eqref{thm:eq:MKEC0}, from Theorem~1.1 of \cite{BB18},
  we obtain that the electromagnetic field  $(E,B)$ belongs to $H_{\rm loc}^\beta(\R^+\times \R^3)$, with
  $\beta= 6/(13+\sqrt{142})$. Setting $p=q=2$ and $\beta= 6/(13+\sqrt{142})$ in the hypotheses of Theorem~\ref{thm:entropies},
  and using assumption \eqref{reg-f-entropy} under constraints \eqref{condab}, we obtain from
  Theorem~\ref{thm:entropies} the desired result.    
\end{namedproof}

\begin{remark}
  From Corollary~\ref{cor:entropies}, we deduce that for $\alpha$ such that,
  \begin{equation}
    \label{eqn:can}
    1> \alpha > \frac{1-\beta}{3+\beta}=\frac{7+\sqrt{142}}{45 +3\sqrt{142}}\simeq 0.234, 
  \end{equation}
  the Vlasov equation \eqref{eqn:RV:2}, which is a first-order conservation law in the phase-space $\R^6$,
  has an infinite number of conserved entropies. A similar situation occurs with general systems
  of conservation laws, which are studied in \cite{BGSGTW18}  within the regularity framework of
  H\"older spaces $\mathscr{C}^{0,\alpha}$.
  Nevertheless in \cite{BGSGTW18}, the
  authors show conservation of entropies under the sufficient condition $\alpha >1/3$ (the famous
  Onsager exponent \cite{Ons49}), which is more restrictive than the present result, from two points of view.
  First,  Sobolev spaces $H^\alpha$ are less regular than   H\"older spaces $\mathscr{C}^{0,\alpha}$,
  for the same $\alpha$. Secondly our index $\alpha$ is smaller than $1/3$.
  An explanation to such  discrepancy, comes from our commutator estimates which exploit the anisotropy between the velocity
  and physical spaces, whereas commutator estimates in \cite{GMSG18, BGSGTW18} use some Taylor expansions, which
  does not advantage a particular direction of space.
  Finally, we observe that the critical exponent $\alpha=(7+\sqrt{142})/(45 +3\sqrt{142})$, which is
  smaller that $\sqrt{5}-2$, can not be retrieved with the method of \cite{Eyi18},
  since the latter is obtained under the condition $p=q\geq 3$ and hence can not deal with the case $p=q=2$.
\end{remark}

\subsection{Proof of Theorem~\ref{thm:entropies}}
\label{ssec:Proof-Thm-entropy}
Before giving the proof of Theorem~\ref{thm:entropies}, we first introduce some standard regularization operators and
we recall their main properties. Using a smooth non-negative function $\varrho$ such that,
\begin{equation}
  \label{prop-molif}
  \tau\mapsto \varrho(\tau) \geq 0, \quad \varrho\in \mathcal{D}(\R), \quad \mbox{supp}(\varrho)\subset]-1,1[, \quad
      \int_\R \varrho(\tau) d\tau=1,
\end{equation}
one define the radially-symmetric compactly-supported Friedrichs mollifier $z\mapsto \varrho_{\epsilon}(z)$, given by
\begin{eqnarray}
  \R^d & \longrightarrow & \R^+ \\
   z  & \longmapsto &\varrho_\epsilon(z)=\frac{1}{\epsilon^d} \varrho\left(\frac{|z|}{\epsilon}\right), \quad \epsilon>0.
\end{eqnarray}
For any distribution $f\in \mathcal{D}'(\R^+\times\R^6)$, we define its $\mathscr{C}^\infty$-regularization by
\begin{equation}
\label{regtxp}
f^{\eta,\varepsilon,\delta}(t,x,{\xi}) =
\varrho_\eta(t) \mathop{\ast}_{t}\varrho_\varepsilon(x) \mathop{\ast}_{x} \varrho_\delta({\xi}) \mathop{\ast}_{{\xi}} f(t,x,{\xi}),
\end{equation}
where the operator $\ast$ denotes the standard convolution product. We denote by $\langle\cdot,\cdot \rangle$ the
dual bracket between spaces $\mathcal{D}'$ and $\mathcal{D}$.
Using previous definitions, we have for the regularization operator $(\cdot)^\epsilon$
the following standard properties (see, e.g., \cite{AG91}), which are summarized in 
\begin{lemma}
  \label{lem:reg_prop}
  \begin{itemize}
    \item[1.] For any distribution $f\in \mathcal{D}'(\R^d)$, we have,
  \begin{equation*}
    \langle f^\epsilon, g \rangle = \langle f, g^\epsilon\rangle, \quad g\in \mathcal{D}(\R^d).
  \end{equation*}
  \item[2.]
  For any function $f\in L^1\cap L^\infty\cap W^{\alpha,p}(\R^d)$, with $0<\alpha <1$ and $1\leq p \leq \infty$, there exists a constant $C$ such that,
  \begin{eqnarray*}
    &&  \|f^\epsilon \|_{L^{q}(\R^d)} \leq  \|f\|_{L^{q}(\R^d)}, \ \ 1\leq q\leq \infty,\\
    && \|f^\epsilon \|_{W^{\alpha,p}(\R^d)} \leq  \|f\|_{W^{\alpha,p}(\R^d)}, \\
    && \|f^\epsilon -f \|_{L^p(\R^d)} \leq 
    C\epsilon^\alpha  \|f\|_{W^{\alpha,p}(\R^d)},\\
    &&  \|\nabla f^\epsilon \|_{L^p(\R^d)} \leq  C\epsilon^{\alpha-1}  \|f\|_{W^{\alpha,p}(\R^d)}.
  \end{eqnarray*}
  \item[3.] For any function $f\in B_{p, \infty}^{\alpha}(\R^d)$, with $0<\alpha <1$ and $1\leq p \leq \infty$, there exists a constant $C$ such that,  
  \begin{equation*}
    \|f(\cdot - z) -f(\cdot) \|_{L^p(\R^d)} \leq C
    |z|^\alpha  \|f\|_{B_{p,\infty}^{\alpha}(\R^d)}.
  \end{equation*}
  \end{itemize}
\end{lemma}  
\begin{proof}
Since the proof is elementary, it is left to the reader.
\end{proof}

In order to prove Theorem~\ref{thm:entropies}, we use some commutator estimates which are given by

\begin{lemma}
\label{lem:commutators}
Let $(f,E,B)$ be a weak solution of the relativistic Vlasov--Maxwell system \eqref{eqn:RV:2}-\eqref{def:gamma},
given by Theorem~\ref{thm:DL89}, satisfying the regularity assumptions
\eqref{reg-feb-entropy-1}-\eqref{reg-feb-entropy-3} of Theorem~\ref{thm:entropies}.
Let us recall that $F_L:=E+ v\times B$
is the Lorentz force field. Then there exist a constant $C_{fs}$ depending on $\|f\|_{L^\infty(0,T;W^{\alpha,p}(\R^6))}$
and a constant $C_{fl}$ depending on  $\|f\|_{L^\infty(0,T;W^{\alpha,p}(\R^6))}$,  $\|E\|_{L^\infty(0,T;W^{\beta,q}(\R^3))}$
and  $\|B\|_{L^\infty(0,T;W^{\beta,q}(\R^3))}$ such that,
\begin{equation}
  \label{commut-fs}
  \big\|\nabla_x \cdot \big( (vf)^{\eta,\varepsilon,\delta} -v^{\delta} f^{\eta,\varepsilon,\delta}
  \big) \big\|_{L^1\left(0,T; L^p(\R^6)\right)} \leq
  C_{fs} \delta^{\alpha +1}\varepsilon^{\alpha-1},
\end{equation}
and
\begin{equation}
  \label{commut-fl}
  \big \|\nabla_{\xi} \cdot \big( (F_Lf)^{\eta,\varepsilon,\delta} -F_L^{\eta,\varepsilon,\delta} f^{\eta,\varepsilon,\delta}
   \big) \big \|_{L^1\left(0,T; L^p(\R_{\xi}^3;L^r\left(\R_x^3)\right)\right)} \leq
    C_{fl} (\varepsilon^{\beta+\alpha}\delta^{\alpha-1} + \delta^\alpha),
\end{equation}
where $\alpha$, $\beta$, $p$, $q$ and $r$ satisfy
 relations \eqref{reg-feb-entropy-1}-\eqref{reg-feb-entropy-2}.
\end{lemma}

\begin{remark}
The precise estimates obtained in Lemma \ref{lem:commutators} seem to be compulsory to obtain the precise Onsager exponents $\alpha,\beta$ in the main theorem, Theorem~\ref{thm:entropies}, instead of the general exponent $1/3$ established for fluid models.  
\end{remark}

\begin{proof}
  We start with two basic estimates, which will be often used along the proof.
  Using  the fundamental theorem of calculus and,
  \begin{equation}
    |\nabla_{\xi} v|=
\Big|    \frac{I_3}{(1+|{\xi}|^2)^{1/2}} - \frac{{\xi}\otimes {\xi}}{(1+|{\xi}|^2)^{3/2}} \Big|\leq \frac{2}{\sqrt{1+|{\xi}|^2}} \leq 2,
    \label{lip-v0}
  \end{equation}
  we obtain the first basic estimate,
  \begin{equation}
    \label{diffv}
    |v({\xi}-w)-v({\xi})| \leq  |w| \int_{0}^1 |\nabla v({\xi}-\tau w)|\, d\tau \leq 2|w|.
  \end{equation}
  Using the fundamental theorem of calculus twice, we obtain componentwise,
  \begin{eqnarray}
    v_i -v_i^\delta &=& \int_{\R^3}dw \,\varrho_\delta(w) ( v_i({\xi})- v_i({\xi}- w)) \nonumber \\
    &=& \sum_j \int_{\R^3}dw \,\varrho_\delta(w) w_j \int_{0}^1 d\tau\, \partial_{j} v_i({\xi}-\tau w) \nonumber \\
    &=&  \sum_j \partial_{j} v_i({\xi}) \int_{\R^3}dw \,\varrho_\delta(w) w_j \nonumber\\
    &&
    +  \sum_{j,k} \int_{\R^3} dw \,\varrho_\delta(w) w_j w_k  \int_0^1d\tau \int_0^1 ds \, \partial_{jk}^2 v_i({\xi}-s\tau w).
    \label{inttv2}
  \end{eqnarray}  
  Since  the smooth function $\varrho$ is radially symmetric and  compactly supported, we have,
  \begin{equation}
    \label{mollifier-prop} 
    \int_{\R^3} dw \,\varrho (w) w_i =0, \quad \mbox{and} \quad
    \int_{\R^3} dw \,\varrho (w) |w_i| |w_j| \leq C_{\varrho} < +\infty, \quad \forall i,j, \in \{1,2,3\},
   \end{equation}
  where  $C_{\varrho}$ is a numerical constant depending only on the function $\varrho$. 
  Using the first equality in \eqref{mollifier-prop}, the first term of the right-hand side of
  \eqref{inttv2} vanishes. Using the second inequality of \eqref{mollifier-prop}, and
  \begin{equation*}
  |  \nabla_{jk}^2 v_i({\xi})|= \Big|\frac{\delta_{ij} {\xi}_k}{(1+|{\xi}|^2)^{3/2}}  +\frac{\delta_{jk} {\xi}_i}{(1+|{\xi}|^2)^{3/2}}
    +\frac{\delta_{ik} {\xi}_j}{(1+|{\xi}|^2)^{3/2}} -
    \frac{3{\xi}_i{\xi}_j{\xi}_k}{(1+|{\xi}|^2)^{3/2}} \Big|\leq  \frac{6}{1+|{\xi}|^2} \leq 6,
    \label{lip2-v0}
  \end{equation*}
  we obtain from \eqref{inttv2} the second basic estimate, 
  \begin{equation}
    |v-v^\delta| \leq 6 C_{\varrho}\delta^2. 
    \label{lip-v}
  \end{equation}
  We now deal with commutator estimate \eqref{commut-fs} for the free-streaming term. We define,
  \begin{multline}
    \label{def-cr}
    r_{\eta,\varepsilon, \delta}(f,g)(t,x,{\xi})=\int_{\R}d\tau\int_{\R^3}dy\int_{\R^3}dw \,
    \varrho_\eta(\tau) \varrho_\varepsilon(y) \varrho_\delta(w)\\
    (f(t-\tau,x-y,{\xi}-w)-f(t,x,{\xi})) (g(t-\tau,x-y,{\xi}-w)-g(t,x,{\xi})).
  \end{multline}
  Using \eqref{def-cr}, it is easy to check that,
  \begin{equation}
    \label{comfs-1}
    (vf)^{\eta,\varepsilon,\delta} = v^\delta f^{\eta,\varepsilon,\delta}
    + r_{\eta,\varepsilon,\delta}(v,f) - (f- f^{\eta,\varepsilon,\delta})(v-v^\delta).
  \end{equation}
  Observing that,
  \[
  r_{\eta,\varepsilon,\delta}(v,f) =  r_{\delta}(v,f^{\eta,\varepsilon}) +  (f- f^{\eta,\varepsilon})(v-v^\delta),
  \]
  Eq. \eqref{comfs-1} becomes,
  \begin{equation}
    \label{comfs-2}
    (vf)^{\eta,\varepsilon,\delta} - v^\delta f^{\eta,\varepsilon,\delta}
    = r_{\delta}(v,f^{\eta,\varepsilon}) - ((f^{\eta,\varepsilon})^\delta - f^{\eta,\varepsilon})(v-v^\delta).
  \end{equation}
Using estimate \eqref{diffv}, Lemma~\ref{lem:reg_prop}, continuous embedding
$W^{\alpha,p}(\R^d)\subset B_{p,\infty}^{\alpha}(\R^d)$ with $1\leq p\leq \infty$,
the restriction property for Sobolev spaces $W^{\alpha,p}(\R^d)$
(see Remark~\ref{restrictions_property}), and regularity assumptions \eqref{reg-feb-entropy-1}-\eqref{reg-feb-entropy-3},
we obtain,
\begin{eqnarray}
  \big \| \nabla_x \cdot r_{\delta}(v,f^{\eta,\varepsilon})\big\|_{L^1(0,T;L^p(\R^6))}
  &\leq& \int_0^T dt \int_{\R^3}dw\,\varrho_\delta(w)\,
  \|(v({\xi}-w)-v({\xi}))\nonumber\\
  && \cdot\, (\nabla_x f^{\eta,\varepsilon}(t,x,\xi-w)
  -\nabla_x f^{\eta,\varepsilon}(t,x,\xi))\|_{L^p(\R_{x\xi}^6)} \nonumber\\
  &\leq& C\int_0^T dt \int_{\R^3}dw\,\varrho_\delta(w)|w|^{\alpha +1}
  \|\nabla_x f^{\eta,\varepsilon}(t) \|_{L^p(\R_x^3;B_{p,\infty}^\alpha(\R_{\xi}^3))} \nonumber\\
  &\leq& C \delta^{\alpha +1}\int_0^T dt\,
  \|\nabla_x f^{\eta,\varepsilon}(t) \|_{L^p(\R_x^3;W^{\alpha,p}(\R_{\xi}^3))} \nonumber\\
  &\leq& C\varepsilon^{\alpha-1}\delta^{\alpha +1}\int_0^T dt\,
  \| f^{\eta}(t) \|_{W^{\alpha, p}(\R_x^3;W^{\alpha,p}(\R_{\xi}^3))} \nonumber\\
  &\leq&C \varepsilon^{\alpha-1}\delta^{\alpha +1}\int_0^T dt \,  \varrho_\eta(t) \ast
  \| f(t) \|_{W^{\alpha, p}(\R^6)} \nonumber\\
  &\leq& C\varepsilon^{\alpha-1}\delta^{\alpha +1}
  \| f\|_{L^1(0,T;W^{\alpha, p}(\R^6))} \label{est:rdvf}.
\end{eqnarray}  
Using estimate \eqref{lip-v}, Lemma~\ref{lem:reg_prop},
the restriction property for Sobolev spaces $W^{\alpha,p}(\R^d)$, and
regularity assumptions \eqref{reg-feb-entropy-1}-\eqref{reg-feb-entropy-3}, we obtain,
\begin{eqnarray}
  \big \| \nabla_x \cdot ((f^{\eta,\varepsilon})^\delta - f^{\eta,\varepsilon})(v-v^\delta))\big\|_{L^1(0,T;L^p(\R^6))}
  &\leq& \|v-v^\delta\|_{L^\infty(\R^3)}
  \big \| (\nabla_xf^{\eta,\varepsilon})^\delta - \nabla_xf^{\eta,\varepsilon}\big\|_{L^1(0,T;L^p(\R^6))} \nonumber\\
  &\leq&C \delta^{\alpha+1}
  \|\nabla_x f^{\eta,\varepsilon} \|_{L^1(0,T; L^p(\R_x^3;W^{\alpha,p}(\R_{\xi}^3)))} \nonumber\\
   &\leq&C \varepsilon^{\alpha-1}\delta^{\alpha+1}
  \| f^{\eta} \|_{L^1(0,T; W^{\alpha,p}(\R_x^3;W^{\alpha,p}(\R_{\xi}^3)))} \nonumber\\
  &\leq & C\varepsilon^{\alpha -1}\delta^{\alpha+1}
  \| f\|_{L^1(0,T;W^{\alpha, p}(\R^6))}. \label{est:dvdf}
\end{eqnarray}
Using \eqref{est:rdvf}-\eqref{est:dvdf}, we obtain from \eqref{comfs-2}, commutator
estimate \eqref{commut-fs}.  We continue with  commutator estimate \eqref{commut-fl}
for the Lorentz force term.   
Using definition \eqref{def-cr}, we first make the following decomposition,
 \begin{equation}
   \label{dec:TETB}
   (F_Lf)^{\eta,\varepsilon,\delta} -F_L^{\eta,\varepsilon,\delta} f^{\eta,\varepsilon,\delta}  = T_E+ T_B,
 \end{equation}
 where
 \begin{equation}
   \label{def:TE}
   T_E= (E f^{\delta})^{\eta,\varepsilon} - E^{\eta,\varepsilon} (f^\delta)^{\eta,\varepsilon}=
     r_{\eta,\varepsilon}(E,f^\delta)-  (E-E^{\eta,\varepsilon}) (f^\delta-  (f^\delta )^{\eta,\varepsilon}),
 \end{equation}
 and
 \begin{equation}
   \label{def:TB}
   T_B= (v\times B f)^{\varepsilon,\delta} - v^\delta \times B^\varepsilon f^{\varepsilon,\delta}.
 \end{equation}
 Let us first deal with the term $T_E$. Passing to the limit $\eta\rightarrow 0$ in $r_{\eta,\varepsilon}(E,f^\delta)$,
 which can be justified by the Lebesgue dominated convergence theorem and regularity assumptions
 \eqref{reg-feb-entropy-1}-\eqref{reg-feb-entropy-3},
 we  obtain, 
\begin{multline}
  \big\|\nabla_\xi\cdot r_{\eta,\varepsilon}(E,f^\delta)\big\|_{L^1(0,T;L^p(\R_\xi^3;L^r(\R_x^3)))}
  \leq  \big\|\nabla_\xi\cdot r_{\varepsilon}(E,f^\delta)\big\|_{L^1(0,T;L^p(\R_\xi^3;L^r(\R_x^3)))}
 \\\leq 
  \int_{\R^3}dy\,\varrho_\varepsilon(y) 
  \|(E(t,x-y)-E(t,x))\\ \cdot\, (\nabla_\xi f^{\delta}(t,x-y,\xi)
  -\nabla_\xi f^{\delta}(t,x,\xi))\|_{L^1(0,T;L^p(\R_\xi^3;L^r(\R_x^3)))}.
  \label{reeefd_1}
\end{multline}
Using  H\"older inequality, Lemma~\ref{lem:reg_prop},  continuous embedding
$W^{\alpha,p}(\R^d)\subset B_{p,\infty}^{\alpha}(\R^d)$ with $1\leq p\leq \infty$,
the restriction property for Sobolev spaces $W^{\alpha,p}(\R^d)$, and
regularity assumptions \eqref{reg-feb-entropy-1}-\eqref{reg-feb-entropy-3},
we obtain from \eqref{reeefd_1},
\begin{eqnarray}
 \big \|\nabla_\xi\cdot r_{\eta,\varepsilon}(E,f^\delta)\big\|_{L^1(0,T;L^p(\R_\xi^3;L^r(\R_x^3)))}
 &\leq& \int_{\R^3}dy\,\varrho_\varepsilon(y)
 \|E(t,x-y)-E(t,x)\|_{L^\infty(0,T;L^q(\R_x^3))} \nonumber\\
 && \|\nabla_\xi f^{\delta}(t,x-y,\xi)
 -\nabla_\xi f^{\delta}(t,x,\xi)\|_{L^1(0,T;L^p(\R_{x\xi}^6))} \nonumber\\
 &\leq& C
 \int_{\R^3}dy\,\varrho_\varepsilon(y) |y|^{\alpha+\beta} \nonumber\\
 && \|E\|_{L^\infty(0,T;B_{q,\infty}^\beta(\R^3))}
  \|\nabla_\xi f^{\delta}\|_{L^1(0,T;L^p(\R_\xi^3;B_{p,\infty}^\alpha(\R_{x}^3)))} \nonumber\\
  &\leq& C\varepsilon^{\alpha+\beta}\|E\|_{L^\infty(0,T;W^{\beta,q}(\R^3))}
  \|\nabla_\xi f^{\delta}\|_{L^1(0,T;L^p(\R_\xi^3;W^{\alpha,p}(\R_{x}^3)))} \nonumber\\
   &\leq& C\varepsilon^{\alpha+\beta} \delta^{\alpha-1}\|E\|_{L^\infty(0,T;W^{\beta,q}(\R^3))}
  \|f\|_{L^1(0,T;W^{\alpha,p}(\R_\xi^3;W^{\alpha,p}(\R_{x}^3)))} \nonumber\\
   &\leq& C\varepsilon^{\alpha+\beta} \delta^{\alpha-1}\|E\|_{L^\infty(0,T;W^{\beta,q}(\R^3))}
  \|f\|_{L^1(0,T;W^{\alpha,p}(\R^6))}. \label{reeefd_2}
\end{eqnarray}
Using the Lebesgue dominated convergence theorem and regularity assumptions
\eqref{reg-feb-entropy-1}-\eqref{reg-feb-entropy-3},
we can pass to the limit $\eta\rightarrow 0$ in the term,
$
(E-E^{\eta,\varepsilon}) (f^\delta-  (f^\delta )^{\eta,\varepsilon}),
$
to obtain, 
\begin{multline}
  \big\|\nabla_\xi\cdot \big((E-E^{\eta,\varepsilon}) (f^\delta-  (f^\delta )^{\eta,\varepsilon})\big) \big\|_{L^1(0,T;L^p(\R_\xi^3;L^r(\R_x^3)))}
  \\\leq \big\|\nabla_\xi\cdot \big ((E-E^{\varepsilon}) (f^\delta-  (f^\delta )^{\varepsilon})\big) \big\|_{L^1(0,T;L^p(\R_\xi^3;L^r(\R_x^3)))}. 
  \label{dEdf_1}
\end{multline}
Using  H\"older inequality, Lemma~\ref{lem:reg_prop},  continuous embedding
$W^{\alpha,p}(\R^d)\subset B_{p,\infty}^{\alpha}(\R^d)$ with $1\leq p\leq \infty$,
the restriction property for Sobolev spaces $W^{\alpha,p}(\R^d)$, and
regularity assumptions \eqref{reg-feb-entropy-1}-\eqref{reg-feb-entropy-3}, we obtain from \eqref{dEdf_1},
\begin{multline}
 \big \|\nabla_\xi\cdot  \big((E-E^{\eta,\varepsilon}) (f^\delta-  (f^\delta )^{\eta,\varepsilon})\big) \big\|_{L^1(0,T;L^p(\R_\xi^3;L^r(\R_x^3)))}
 \\\leq \|E-E^\varepsilon\|_{L^\infty(0,T;L^q(\R^3))} \|\nabla_\xi f^{\delta}
 -(\nabla_\xi f^{\delta})^\varepsilon\|_{L^1(0,T;L^p(\R^6))} 
 \\\leq C \varepsilon^{\alpha+\beta} \|E\|_{L^\infty(0,T;B_{q,\infty}^\beta(\R^3))}
 \|\nabla_\xi f^{\delta}\|_{L^1(0,T;L^p(\R_\xi^3;B_{p,\infty}^\alpha(\R_{x}^3)))} 
   \\\leq C\varepsilon^{\alpha+\beta} \delta^{\alpha-1}\|E\|_{L^\infty(0,T;W^{\beta,q}(\R^3))}
  \|f\|_{L^1(0,T;W^{\alpha,p}(\R_\xi^3;W^{\alpha,p}(\R_{x}^3)))} 
   \\\leq C\varepsilon^{\alpha+\beta} \delta^{\alpha-1}\|E\|_{L^\infty(0,T;W^{\beta,q}(\R^3))}
  \|f\|_{L^1(0,T;W^{\alpha,p}(\R^6))}. \label{dEdf_2}
\end{multline}  
From \eqref{reeefd_1} and \eqref{dEdf_2}, we obtain,
\begin{equation}
\label{divTE}
\big \|\nabla_\xi\cdot T_E \big\|_{L^1(0,T;L^p(\R_x^3;L^r(\R_x^3)))}
\leq C\varepsilon^{\alpha+\beta} \delta^{\alpha-1}\|E\|_{L^\infty(0,T;W^{\beta,q}(\R^3))}
  \|f\|_{L^1(0,T;W^{\alpha,p}(\R^6))}. 
\end{equation}  
We now deal with the Term $T_B$, given by \eqref{def:TB}, and  which can be recast as,
\begin{multline}
  T_B= \int_0^T d\tau \int_{\R^3} dy \int_{\R^3} dw \, \varrho_{\eta}(\tau) \varrho_{\varepsilon}(y)\varrho_{\delta}(w)
  [v({\xi}-w)-v({\xi})] \times B(t-\tau,x-y) f(t-\tau,x-y,{\xi}-w) \\ +\  v\times
   [ (B f^\delta )^{\eta,\varepsilon} - B^{\eta,\varepsilon}  (f^\delta )^{\eta,\varepsilon} ]
  \ +\   (v-v^\delta )\times B^{\eta,\varepsilon} (f^\delta)^{\eta,\varepsilon} = T_{B1} + T_{B2} + T_{B3} \label{def:TB123}.
\end{multline}
The term $\nabla_\xi \cdot T_{B1}$ can be decomposed as,
\begin{multline}
\label{divTB1}
\nabla_\xi \cdot T_{B1}= \\
\int_0^T d\tau\int_{\R^3} dy \int_{\R^3} dw \, \varrho_{\eta}(\tau)\varrho_{\varepsilon}(y) \nabla_w\varrho_{\delta}(w) \cdot
    ([v({\xi}-w)-v({\xi})] \times B(t-\tau,x-y)) f(t-\tau,x-y,{\xi})
    + \\ 
\int_0^T d\tau
 \int_{\R^3} dy \int_{\R^3} dw \, \varrho_{\eta}(\tau)\varrho_{\varepsilon}(y) \nabla_w\varrho_{\delta}(w) \cdot
     ([v({\xi}-w)-v({\xi})] \times B(t-\tau,x-y)) \\ (f(t-\tau,x-y,{\xi}-w)-f(t-\tau,x-y,\xi))
     =T_{B11} + T_{B12}    .
\end{multline}
Using integration by parts, we observe that,
\begin{multline}
T_{B11} =\int_0^T d\tau\int_{\R^3} dy \int_{\R^3} dw \, \varrho_{\eta}(\tau)\varrho_{\varepsilon}(y) \varrho_{\delta}(w)
\\\nabla_w \cdot ([v({\xi}-w)-v({\xi})] \times B(t-\tau,x-y)) f(t-\tau,x-y,{\xi}) =0,
\label{divTB11}
\end{multline}  
because $\nabla_w \cdot ([v({\xi}-w)-v({\xi})] \times B(t,x-y))=0$.
Using  H\"older inequality, estimate \eqref{diffv}, Lemma~\ref{lem:reg_prop},  continuous embedding
$W^{\alpha,p}(\R^d)\subset B_{p,\infty}^{\alpha}(\R^d)$ with $1\leq p\leq \infty$,
the restriction property for Sobolev spaces $W^{\alpha,p}(\R^d)$, and
regularity assumptions \eqref{reg-feb-entropy-1}-\eqref{reg-feb-entropy-3}, we obtain,
\begin{multline}
  \big \| T_{B12} \big\|_{L^1(0,T;L^p(\R_\xi^3;L^r(\R_x^3)))}
 \leq
2\int_0^T d\tau\int_{\R^3} dy \int_{\R^3} dw \, \varrho_{\eta}(\tau)\varrho_{\varepsilon}(y) |\nabla_w\varrho_{\delta}(w)|
|w| \\ \|B(t-\tau,x-y) (f(t-\tau,x-y,{\xi}-w)-f(t-\tau,x-y,\xi)) \|_{L^1(0,T;L^p(\R_\xi^3;L^r(\R_x^3)))} 
\\ \leq
2\int_0^T d\tau\int_{\R^3} dy \int_{\R^3} dw \, \varrho_{\eta}(\tau)\varrho_{\varepsilon}(y) |\nabla_w\varrho_{\delta}(w)|
|w| \\\|B(t-\tau,x-y)\|_{L^\infty(0,T;L^q(\R_x^3))} \|f(t-\tau,x-y,{\xi}-w)-f(t-\tau,x-y,\xi) \|_{L^1(0,T;L^p(\R_{x\xi}^6))} 
\\ \leq
C \int_{\R^3} dw \,|\nabla_w\varrho_{\delta}(w)|
|w|^{\alpha+1}\|B\|_{L^\infty(0,T;L^q(\R_x^3))} \|f \|_{L^1(0,T;L^p(\R_{x}^3;B_{p,\infty}^\alpha(\R_\xi^3)))} 
\\ \leq
 C \delta^{\alpha}\|B\|_{L^\infty(0,T;W^{\beta,q}(\R^3))}
\|f\|_{L^1(0,T;W^{\alpha,p}(\R^6))}. \label{divTB12}
\end{multline}
In the similar way we have obtained estimate \eqref{divTE} for $\nabla_\xi \cdot T_E$, we also obtain for $\nabla_\xi \cdot T_{B2}$,
\begin{equation}
\label{divTB2}
\big \|\nabla_\xi\cdot T_{B2} \big\|_{L^1(0,T;L^p(\R_\xi^3;L^r(\R_x^3)))}
\leq C\varepsilon^{\alpha+\beta} \delta^{\alpha-1}\|B\|_{L^\infty(0,T;W^{\beta,q}(\R^3))}
  \|f\|_{L^1(0,T;W^{\alpha,p}(\R^6))}. 
\end{equation}  
Using estimate \eqref{lip-v}, H\"older inequality and Lemma~\ref{lem:reg_prop}, we obtain for  $\nabla_\xi \cdot T_{B3}$,
\begin{eqnarray}
\big \| \nabla_\xi \cdot T_{B3} \big\|_{L^1(0,T;L^p(\R_\xi^3;L^r(\R_x^3)))} &\leq&
\big \| (v-v^\delta )\times B^{\eta,\varepsilon} (\nabla_\xi f^\delta)^{\eta,\varepsilon} \big\|_{L^1(0,T;L^p(\R_\xi^3;L^r(\R_x^3)))}\nonumber\\
&\leq& C\delta^2 \|B^{\eta,\varepsilon}\|_{L^\infty(0,T;L^q(\R_x^3))} \| \nabla_\xi f^{\eta,\delta, \varepsilon}\|_{L^1(0,T;L^p(\R^6))} \nonumber\\
&\leq& C\delta^{\alpha+1}\|B\|_{L^\infty(0,T;L^q(\R_x^3))} \|f\|_{L^1(0,T;W^{\alpha,p}(\R^6))}.
\label{divTB3}
\end{eqnarray}
Gathering estimates \eqref{divTB11}-\eqref{divTB3}, we obtain from decompositions \eqref{def:TB123}-\eqref{divTB1},
\begin{equation}
\label{divTB}
\big \|\nabla_\xi\cdot T_B\big\|_{L^1(0,T;L^p(\R_x^3;L^r(\R_x^3)))}
\leq C(\varepsilon^{\alpha+\beta} \delta^{\alpha-1} + \delta^{\alpha})\|B\|_{L^\infty(0,T;W^{\beta,q}(\R^3))}
  \|f\|_{L^1(0,T;W^{\alpha,p}(\R^6))}. 
\end{equation}
Eventually, from \eqref{divTE} and \eqref{divTB}, we obtain  
commutator estimate \eqref{commut-fl}, which ends the proof of Lemma~\ref{lem:commutators}
\end{proof}

\begin{namedproof} {\it of Theorem}~\ref{thm:entropies}. Let us now give the proof of the main theorem. The weak formulation for the Vlasov equation reads, 
\begin{equation}
\label{weak-vlasov}
\int_0^T dt \int_{\R^3} dx \int_{\R^3} d{\xi}\, f (\partial_t \Psi + v\cdot \nabla_x \Psi + F_L \cdot \nabla_{\xi} \Psi) = 0, \quad \forall \Psi\in
\mathcal{D}((0,T)\times \R^6), 
\end{equation}
with $F_L:=E + v\times B$.  Let us note that all integrals in \eqref{weak-vlasov} have a sense since
for DiPerna--Lions weak solutions \cite{DL89a} we have $f\in L^\infty(0,T; L^2(\R^6))$, and 
$E,\, B \in L^\infty(0,T; L^2(\R^3))$. We choose in \eqref{weak-vlasov} the test function,
\begin{equation}
  \label{test-function}
  \Psi = \Psi_{\varepsilon,\delta}=(\mathcal{H}'(f^{\eta,\varepsilon,\delta}) \Phi)^{\eta,\varepsilon,\delta} \in \mathcal{D}((0,T)\times \R^6),
\end{equation}
 with $\Phi\in \mathcal{D}((0,T)\times \R^6)$  and $\mathcal{H}\in \mathscr{C}^1(\R^+;\R^+)$.
 Using the first property of Lemma~\ref{lem:reg_prop}  and successive integrations by parts,
 we obtain from \eqref{weak-vlasov}-\eqref{test-function},
\begin{multline}
  \label{weak-entropy-reg}
  \int_0^T dt \int_{\R^3} dx \int_{\R^3} d{\xi} \, \Big\{ \mathcal{H}(f^{\eta,\varepsilon,\delta}) (\partial_t \Phi + v^\delta\cdot \nabla_x
  \Phi + F_L^{\eta,\varepsilon,\delta} \cdot \nabla_{\xi} \Phi) \ 
 \\ 
 + \ \Phi\mathcal{H}'(f^{\eta,\varepsilon,\delta})\Big[\nabla_x \cdot \big( (vf)^{\eta,\varepsilon,\delta}-v^\delta f^{\eta,\varepsilon,\delta}\big)
    +\nabla_{\xi} \cdot \big( (F_Lf)^{\eta,\varepsilon,\delta}-  F_L^{\eta,\varepsilon,\delta}f^{\eta,\varepsilon,\delta}\big)\Big]\Big \} =0, 
\end{multline}
for all $\Phi\in \mathcal{D}((0,T)\times \R^6)$.
We now establish the renormalized Vlasov equation \eqref{entropeq_R}.
Using  regularity assumptions \eqref{reg-feb-entropy-1}-\eqref{reg-feb-entropy-3},
\eqref{test-function}, Lemma~\ref{lem:reg_prop}~and~\ref{lem:commutators},
we obtain from \eqref{weak-entropy-reg},
\begin{multline}
\label{final-local-entropy-est}
\left | \int_0^T dt \int_{\R^3} dx \int_{\R^3} d{\xi} \,  \mathcal{H}(f^{\eta,\varepsilon,\delta})
(\partial_t \Phi +  v^\delta\cdot \nabla_x
  \Phi + F_L^{\eta,\varepsilon,\delta} \cdot \nabla_{\xi} \Phi)
  \right| 
\\
\leq  C_{\ast} \left(\delta^{\alpha +1}\varepsilon^{\alpha-1} +  \varepsilon^{\alpha+\beta}\delta^{\alpha-1}+\delta^\alpha\right),
\end{multline}
where  $ C_{\ast} $ depends on  $\|f\|_{L^\infty(0,T;L^{\infty}(\R^6))}$, $C_{fs}$,   $C_{fl}$, $\mathcal{H}$, and $\Phi$.
Balancing contributions coming from the free-streaming and Lorentz force terms in the right-hand side
of \eqref{final-local-entropy-est}, we obtain,
\begin{equation}
  \label{eqn:2deg}
\varepsilon^{\alpha -1}\delta^{2}  -\delta -\varepsilon^{\alpha+\beta}=0,
\end{equation}
and estimate \eqref{final-local-entropy-est} becomes,
\begin{equation}
\label{final-local-entropy-est-2}
\left | \int_0^T dt \int_{\R^3} dx \int_{\R^3} d{\xi} \,  \mathcal{H}(f^{\eta,\varepsilon,\delta})
(\partial_t \Phi +  v^\delta\cdot \nabla_x
  \Phi + F_L^{\eta,\varepsilon,\delta} \cdot \nabla_{\xi} \Phi)
\right|
\leq  C_{\ast}\eta,
\end{equation}
with the definition,
\[
\eta:=\varepsilon^{\alpha-1}\delta^{\alpha+1}.
\]
Solving quadratic equation \eqref{eqn:2deg} in $\delta$, the only positive solution is given by,
\begin{equation*}
\label{eqn:2degSol}
\delta = \frac{1 +\sqrt{1+4\varepsilon^{2\alpha +\beta -1}}}{2\varepsilon^{\alpha -1}}.
\end{equation*}
Two cases are to be considered according to the value of $\alpha$ and $\beta$:
\begin{itemize}
\item[i)] $2\alpha +\beta -1 <0$. We then have $\delta \simeq \varepsilon^{(\beta+1)/2}$, and
  $\eta\simeq\varepsilon^{(\alpha\beta+\beta+3\alpha-1)/2}\rightarrow 0$ as $\varepsilon \rightarrow 0$
  if $\alpha\beta + \beta +3\alpha -1 > 0$.
\item[ii)]  $2\alpha +\beta -1\geq 0$. We then have $\delta \simeq \varepsilon^{1-\alpha}$, and
  $\eta\simeq\varepsilon^{\alpha(1-\alpha)}\rightarrow 0$ as $\varepsilon \rightarrow 0$
  if $0<\alpha <1$.
\end{itemize}  
Assuming that the free-streaming contribution dominates the Lorentz-force contribution, this 
implies $\varepsilon^{\alpha -1}\delta^{2}  -\delta -\varepsilon^{\alpha+\beta} \gg 0$, which leads to a contradiction
as $\delta\rightarrow 0$.  On the contrary, assuming that the  Lorentz-force  contribution dominates the free-streaming  contribution,
this implies $\varepsilon^{\alpha -1}\delta^{2}  -\delta -\varepsilon^{\alpha+\beta} \ll 0$, which leads also to a contradiction
as first $\delta\rightarrow 0$ and next $\varepsilon \rightarrow 0$.  
In conclusion, if $\alpha\beta + \beta +3\alpha -1 > 0$, then the right-hand side of \eqref{final-local-entropy-est-2}
vanishes as $(\varepsilon, \delta)\rightarrow 0$, and we
obtain the renormalized Vlasov equation \eqref{entropeq_R}.

We continue with the local-in-space  entropy conservation law \eqref{entropeq_L1}. For this purpose,
we first restrict  entropy functions $\mathcal{H}$ to the set $\mathscr{E}$, defined by \eqref{entropy_space}, and secondly
we take in \eqref{final-local-entropy-est-2} a test function $\Phi$ such that,
\[
\Phi(t,x,{\xi})= \Lambda(t,x) \Theta({\xi}), \  \  \mbox{ with } \
\Lambda\in \mathcal{D}((0,T)\times \R^3), \ \mbox{ and } \ \Theta \in \mathcal{D}(\R^3).
\]
We then choose the test function $\Theta$ such that,
\[
\Theta({\xi})= \Theta_R({\xi}) := \theta({\xi}/R), \ \ \mbox{with } \ R>0.
\]
Here the function $\theta \in \mathcal{D}(\R^3)$ is such that 
${\rm supp}(\theta) \subset B_{\R^3}(0,2)$,
$\theta \equiv 1$ on  $B_{\R^3}(0,1)$ and
$0\leq\theta \leq 1$ on  $B_{\R^3}(0,2)\setminus B_{\R^3}(0,1)$.
We then have,
\begin{equation}
  \label{test-PhiR0}
  \Theta_R \longrightarrow 1, \ \ {\rm a.e} \ \  {\rm as} \ \  R\rightarrow +\infty, \ \mbox{ and } \ 
  \nabla_{{\xi}}\Theta_R \longrightarrow 0, \ \ {\rm a.e} \ \  {\rm as} \ \  R\rightarrow +\infty.
\end{equation}
From the uniform integrability assumption \eqref{H:UI}, and the de La Vall\'ee Poussin theorem, there
exists a constant $C_{\mathcal{H}}>0$, independent of  $(\varepsilon,\delta)$, but depending on $\mathcal{H}$ such that, 
\begin{equation}
\label{UI-DVP}
\int_{\R^3} dx \int_{\R^3} d{\xi}\, \mathcal{H}(f^{\eta,\varepsilon,\delta}) \leq C_{\mathcal{H}} <+\infty, \quad \forall \mathcal{H} \in \mathscr{E}.
\end{equation}  
Using estimate \eqref{UI-DVP}, regularity assumptions \eqref{reg-feb-entropy-1}-\eqref{reg-feb-entropy-3}
and property \eqref{test-PhiR0}, we obtain 
from the Lebesgue dominated convergence theorem that,
\begin{equation}
\label{eqn:T1}
 \int_0^T dt \int_{\R^3} dx \int_{\R^3} d{\xi} \,\mathcal{H}(f^{\eta,\varepsilon,\delta}) \partial_t\Lambda \Theta_R
\longrightarrow \int_0^T dt \int_{\R^3} dx \int_{\R^3} d{\xi} \,\mathcal{H}(f^{\eta,\varepsilon,\delta}) \partial_t\Lambda, 
\ \ {\rm as} \ \  R\rightarrow +\infty,
\end{equation}
\begin{equation}
\label{eqn:T2}
 \int_0^T dt \int_{\R^3} dx \int_{\R^3} d{\xi} \,\mathcal{H}(f^{\eta,\varepsilon,\delta}) v^\delta\cdot \nabla_x
\Lambda \Theta_R \longrightarrow \int_0^T dt \int_{\R^3} dx \int_{\R^3} d{\xi} \,\mathcal{H}(f^{\eta,\varepsilon,\delta}) v^\delta\cdot \nabla_x
\Lambda, \ \ {\rm as} \ \  R\rightarrow +\infty,
\end{equation}
and
\begin{equation}
\label{eqn:T3}
\mathcal{R}_1:= \int_0^T dt \int_{\R^3} dx \int_{\R^3} d{\xi} \, \mathcal{H}(f^{\eta,\varepsilon,\delta})
F_L^{\eta,\varepsilon,\delta}f^{\eta,\varepsilon,\delta}\cdot \nabla_{\xi} \Theta_R \Lambda 
\longrightarrow 0\, \ \ {\rm as} \ \  R\rightarrow +\infty.
\end{equation}
Limits \eqref{eqn:T1}-\eqref{eqn:T3} are uniform in $(\eta,\varepsilon, \delta)$, and in addition
there exists a constant $\kappa_1>0$, independent of $(\eta,\varepsilon,\delta)$, but depending
on $\|f\|_{L^\infty(0,T;L^2\cap L^\infty(\R^6))}$,  $\|B\|_{L^\infty(0,T;L^2(\R^3))}$,  $\|E\|_{L^\infty(0,T;L^2(\R^3))}$,
$\Lambda$ and $\theta$ such
that,
\begin{equation}
\label{eqn:T3b}
 |\mathcal{R}_1| \leq \kappa_1 R^{-1}.   
\end{equation}
Using \eqref{eqn:T1}-\eqref{eqn:T3b}, we obtain from
\eqref{final-local-entropy-est-2},
\begin{equation}
\label{final-local-entropy-est-3}
\left | \int_0^T dt \int_{\R^3} dx\ (\partial_t \Lambda + v^\delta\cdot\nabla_x \Lambda)\int_{\R^3} d{\xi} \,  \mathcal{H}(f^{\eta,\varepsilon,\delta})
\right|
\leq  C_{\ast} \eta + \kappa_1 R^{-1}.
\end{equation}
Under the condition, $\alpha\beta + \beta +3\alpha -1 > 0$, the right-hand side of \eqref{final-local-entropy-est-3}
vanishes as $(\eta,\varepsilon, \delta)\rightarrow 0$ and $R\rightarrow +\infty$, and we obtain from
\eqref{final-local-entropy-est-3} the local-in-space conservation law \eqref{entropeq_L1}.
In a similar way, by interchanging the role of the test functions $\Lambda$ and $\Theta$, we obtain
local-in-momentun conservation law \eqref{entropeq_L2}.

We pursue with global entropy conservation law \eqref{entropeq_G}. For this aim,
we first take in \eqref{final-local-entropy-est-3} a test function $\Lambda$ such that,
\[
\Lambda(t,x)= \varphi(t)\Lambda(x), \  \  \mbox{ with } \
\varphi\in \mathcal{D}((0,T)), \ \mbox{ and } \ \Lambda \in \mathcal{D}(\R^3).
\]
We then choose the test function $\Lambda$ such that,
\[
\Lambda(x)= \Lambda_R(x) := \lambda(x/R), \ \ \mbox{with } \ R>0.
\]
Here the function $\lambda \in \mathcal{D}(\R^3)$ is such that 
${\rm supp}(\lambda) \subset B_{\R^3}(0,2)$,
$\lambda \equiv 1$ on  $ B_{\R^3}(0,1)$ and
$0\leq\lambda \leq 1$ on  $(B_{\R^3}(0,2)\setminus B_{\R^3}(0,1))$.
We then have,
\begin{equation}
  \label{test-LambdaR0}
  \Lambda_R \longrightarrow 1, \ \ {\rm a.e} \ \  {\rm as} \ \  R\rightarrow +\infty, \ \mbox{ and } \ 
  \nabla_{x}\Lambda_R \longrightarrow 0, \ \ {\rm a.e} \ \  {\rm as} \ \  R\rightarrow +\infty.
\end{equation}
Using estimate \eqref{UI-DVP}, regularity assumptions \eqref{reg-feb-entropy-1}-\eqref{reg-feb-entropy-3}
and property \eqref{test-LambdaR0}, we obtain
from the Lebesgue dominated convergence theorem that,
\begin{equation}
\label{eqn:T4}
 \int_0^T dt \int_{\R^3} dx \int_{\R^3} d{\xi} \,\mathcal{H}(f^{\eta,\varepsilon,\delta}) \partial_t\varphi \Lambda_R
\longrightarrow \int_0^T dt \int_{\R^3} dx \int_{\R^3} d{\xi} \,\mathcal{H}(f^{\eta,\varepsilon,\delta}) \partial_t\varphi, 
\ \ {\rm as} \ \  R\rightarrow +\infty,
\end{equation}
\begin{equation}
\label{eqn:T5}
\mathcal{R}_2:= \int_0^T dt \int_{\R^3} dx \int_{\R^3} d{\xi} \,\mathcal{H}(f^{\eta,\varepsilon,\delta}) v^\delta\cdot \nabla_x
\Lambda_R \varphi \longrightarrow 0, \ \ {\rm as} \ \  R\rightarrow +\infty.
\end{equation}
Limits \eqref{eqn:T4}-\eqref{eqn:T5} are uniform in $(\eta,\varepsilon, \delta)$, and in addition
there exists a constant $\kappa_2>0$, independent of $(\eta,\varepsilon,\delta)$, but depending
on $C_{\mathcal{H}}$, $\varphi$, and $\lambda$ such
that,
\begin{equation}
\label{eqn:T5b}
 |\mathcal{R}_2| \leq \kappa_2 R^{-1}.   
\end{equation}
Using \eqref{eqn:T4}-\eqref{eqn:T5b}, we obtain from
\eqref{final-local-entropy-est-3},
\begin{equation}
\label{final-global-entropy-est}
\left | \int_0^T dt\, \partial_t \varphi \int_{\R^3} dx \int_{\R^3} d{\xi} \,  \mathcal{H}(f^{\eta,\varepsilon,\delta})
 \right|
\leq  C_{\ast} \eta + \kappa_1 R^{-1} + \kappa_2 R^{-1}.
\end{equation}
Under the condition, $\alpha\beta + \beta +3\alpha -1 > 0$, the right-hand side of \eqref{final-global-entropy-est}
vanishes as $(\eta,\varepsilon, \delta)\rightarrow 0$ and $R\rightarrow +\infty$, and we obtain from
\eqref{final-global-entropy-est} the global entropy conservation law \eqref{entropeq_G}.
This ends the proof of Theorem~\ref{thm:entropies}
\end{namedproof}

\section{Energy conservation}
As concerns conservation of total energy  we have,
\begin{theorem}
   \label{thm:EC1}
  Let $(f,E,B)$ be a weak solution to
  the relativistic Vlasov--Maxwell system \eqref{eqn:RV:2}-\eqref{def:gamma}, given by Theorem~\ref{thm:DL89}.
  If the macroscopic kinetic energy density satisfies the supplementary integrability condition,
  \begin{equation}
    \label{thm:eq:MKEC1}
    \int_{\R^3}  \gamma f \,d{\xi}\in L^\infty\big(0,T; L^2(\R^3)\big), 
  \end{equation}
  then, using definition \eqref{thm:eq:IC2}, we have the local conservation law of total energy,
  \begin{equation}
    \label{TEC_L}
    \partial_t\mathcal{E} + \nabla \cdot \left(\int_{\R^3}  \gamma  f v \,d{\xi}   + E\times B \right)=0, \quad
    \mbox{ in } \  \mathcal{D}'((0,T)\times \R^3),
  \end{equation}
  and the global conservation law of total energy,
  \begin{equation}
    \label{TEC_G}
    \mathcal{E}(t) =\mathcal{E}(s), \ \  \mbox{ for } \  0<s\leq t<T.
  \end{equation}
\end{theorem}

\begin{remark}
  Under assumption \eqref{thm:eq:MKEC1}, it has been proved in \cite{BB18} that the electromagnetic field $(E,B)$
  belongs to $H_{\rm loc}^s(\R_\ast^+\times\R^3)$, with $s=6/(13+\sqrt{142})$. Then, such solutions satisfy
  the conservation laws \eqref{TEC_L}-\eqref{TEC_G}.
\end{remark}

\begin{proof}
Choosing in the weak formulation \eqref{weak-vlasov} the test function,
\begin{equation}
  \label{testfunctionE}
  \Psi(t,x,{\xi}) = \Lambda(t,x)\Theta({\xi}) \gamma({\xi})\in \mathcal{D}((0,T)\times \R^6), \ \  \mbox{ with } \ 
  \Lambda\in \mathcal{D}((0,T)\times \R^3), \  \mbox{ and}  \ \Theta \in \mathcal{D}(\R^3), 
\end{equation}
and using  $\nabla_{\xi} \gamma= v$,  we obtain,
\begin{multline}
  \label{weakenergy}
  \int_0^T dt \int_{\R^3} dx  \left(\int_{\R^3} d{\xi}\, \gamma f \Theta\right) \partial_t \Lambda
  +  \int_0^T dt \int_{\R^3} dx  \left(\int_{\R^3} d{\xi}\, \gamma  f v  \Theta\right) \cdot \nabla_x \Lambda
\\  + \int_0^T dt \int_{\R^3} dx  \left(\int_{\R^3} d{\xi}\,  f v \cdot E \Theta\right)  \Lambda
  + \int_0^T dt \int_{\R^3} dx  \left(\int_{\R^3} d{\xi}\, \gamma f F_L\cdot\nabla_{\xi}\Theta\right)  \Lambda =0.
\end{multline}
We now  establish the local conservation law of total energy.
For this we take in \eqref{weakenergy} a test function $\Theta$ such that,
\[
\Theta({\xi})= \Theta_R({\xi}) := \theta({\xi}/R), \ \ \mbox{with } \ R>0.
\]
Here the function $\theta \in \mathcal{D}(\R^3)$ is such that 
${\rm supp}(\theta) \subset B_{\R^3}(0,2)$,
$\theta \equiv 1$ on  $B_{\R^3}(0,1)$ and
$0\leq\theta \leq 1$ on  $B_{\R^3}(0,2)\setminus B_{\R^3}(0,1)$.
We then have,
\begin{equation}
  \label{testThetaR}
  \Theta_R \longrightarrow 1, \ \ {\rm a.e} \ \  {\rm as} \ \  R\rightarrow +\infty, \ \mbox{ and } \ 
  \nabla_{{\xi}}\Theta_R \longrightarrow 0, \ \ {\rm a.e} \ \  {\rm as} \ \  R\rightarrow +\infty.
\end{equation}  
Using \eqref{testThetaR} and regularity properties \eqref{reg_prop_wsrvm},
we obtain from the Lebesgue dominated convergence theorem,
\begin{equation}
  \label{eq:P1}
  \int_0^T dt \int_{\R^3} dx  \left(\int_{\R^3} d{\xi}\, \gamma f \Theta_R\right) \partial_t \Lambda
  \longrightarrow
  \int_0^T dt \int_{\R^3} dx  \left(\int_{\R^3} d{\xi}\, \gamma f\right) \partial_t \Lambda,\ \mbox{ as } R\rightarrow \infty, 
\end{equation}
and
\begin{equation}
  \label{eq:P2}
  \int_0^T dt \int_{\R^3} dx  \left(\int_{\R^3} d{\xi}\, \gamma  f v  \Theta_R\right) \cdot \nabla_x \Lambda
  \longrightarrow
     \int_0^T dt \int_{\R^3} dx  \left(\int_{\R^3} d{\xi}\, \gamma  f v\right) \cdot \nabla_x \Lambda,\ \mbox{ as } R\rightarrow \infty. 
\end{equation}
Using assumption \eqref{thm:eq:MKEC1}, regularity properties \eqref{reg_prop_wsrvm} and H\"older inequality, we obtain,
\begin{multline}
\left|
\int_0^T dt \int_{\R^3} dx  \left(\int_{\R^3} d{\xi}\, \gamma f F_L\cdot\nabla_{\xi}\Theta_R\right)  \Lambda
\right|
 \leq C  R^{-1} \|\nabla \theta \|_{L^\infty}\|\Lambda \|_{L^\infty} \\
\Big\| \int_{\R^3} d{\xi}\, \gamma f\Big\|_{L^\infty(0,T;L^{r'}(\R^3))} (\| E\|_{L^\infty(0,T;L^r(\R^3))} + \| B\|_{L^\infty(0,T;L^r(\R^3))}) \longrightarrow 0, \
\mbox{ as } R\rightarrow \infty,
\label{eq:P3}
\end{multline}
with $1/r+1/r'=1$ and setting $r=2$.
We now claim that,
\begin{equation}
  \label{domL1}
  |f v \cdot E \Lambda| \leq |\Lambda| |E| f \  \in L^\infty(0,T;L^1(\R^6)) \ \  \mbox{ if }  \ \ \int_{\R^3} d{\xi}\, \gamma f \in
   L^\infty(0,T;L^{3/2}(\R^3)).
\end{equation}
Indeed using interpolation Lemma~2.3 in \cite{BB18}, we obtain,
\begin{equation}
  \label{interplem}
  \Big \| \int_{\R^3} d{\xi}\, f  \Big \|_{L^\infty(0,T;L^2(\R^3))}
  \leq 9\|f\|_{L^\infty}^{1/4}  \Big \| \int_{\R^3} d{\xi}\,\gamma f \Big \|_{L^\infty(0,T;L^{3/2}(\R^3))}^{3/4}.
\end{equation}
Therefore  \eqref{domL1} results from \eqref{interplem} and Cauchy-Schwarz inequality. We notice that the
$L^{3/2}$-integrability condition in \eqref{domL1} results from regularity properties \eqref{reg_prop_wsrvm}, assumption \eqref{thm:eq:MKEC1}
and standard interpolation results between Lebesgue spaces.
Using \eqref{testThetaR} we have $f v \cdot E \Theta_R \Lambda \rightarrow f v \cdot E \Lambda\ $ a.e. as $R\rightarrow +\infty$.
Moreover using \eqref{domL1}, we obtain, from the Lebesgue dominated convergence theorem,
\begin{equation}
  \label{eq:P4}
  \int_0^T dt \int_{\R^3} dx  \left(\int_{\R^3} d{\xi}\,  f v \cdot E \Theta_R\right)  \Lambda
  \longrightarrow  \int_0^T dt \int_{\R^3} dx \,  j \cdot E \Lambda, \ \mbox{ as } R\rightarrow \infty.
\end{equation}  
Using the weak formulation of the Maxwell equation, we obtain, 
\begin{equation}
  \label{eq:P5}
  \int_0^T dt \int_{\R^3} dx\,   j \cdot E \Lambda =
    \int_0^T dt \int_{\R^3} dx\, \frac{|E|^2 + |B|^2}{2}\,\partial_t\Lambda +  \int_0^T dt \int_{\R^3} dx\, E\times B\cdot\nabla_x\Lambda.  
\end{equation}  
Using \eqref{eq:P1}-\eqref{eq:P3} and  \eqref{eq:P4}-\eqref{eq:P5}, we obtain from \eqref{weakenergy},
\begin{multline}
 \label{weakenergy2}
  \int_0^T dt \int_{\R^3} dx  \left\{ \left(\int_{\R^3} d{\xi}\, \gamma f\right)  + \frac{|E|^2 + |B|^2}{2} \right\}\partial_t \Lambda
  \\ +  \int_0^T dt \int_{\R^3} dx
  \left\{ \left(\int_{\R^3} d{\xi}\, \gamma  f v \right)  + E\times B \right \}\cdot \nabla_x \Lambda
  =0,
\end{multline}
which gives the local conservation law of total energy \eqref{TEC_L}. We continue by deriving the global
conservation law of total energy. For this we take in \eqref{weakenergy2} a test function $\Lambda$ such that,
\[
\Lambda(t,x)= \varphi(t)\Lambda(x), \  \  \mbox{ with } \
\varphi\in \mathcal{D}((0,T)), \ \mbox{ and } \ \Lambda \in \mathcal{D}(\R^3).
\]
We then choose the test function $\Lambda$ such that,
\[
\Lambda(x)= \Lambda_R(x) := \lambda(x/R), \ \ \mbox{with } \ R>0.
\]
Here the function $\lambda \in \mathcal{D}(\R^3)$ is such that 
${\rm supp}(\lambda) \subset B_{\R^3}(0,2)$,
$\lambda \equiv 1$ on  $ B_{\R^3}(0,1)$ and
$0\leq\lambda \leq 1$ on  $(B_{\R^3}(0,2)\setminus B_{\R^3}(0,1))$.
We then have,
\begin{equation}
  \label{testLambdaR}
  \Lambda_R \longrightarrow 1, \ \ {\rm a.e} \ \  {\rm as} \ \  R\rightarrow +\infty, \ \mbox{ and } \ 
  \nabla_{x}\Lambda_R \longrightarrow 0, \ \ {\rm a.e} \ \  {\rm as} \ \  R\rightarrow +\infty.
\end{equation}  
Using \eqref{testLambdaR}  and regularity properties \eqref{reg_prop_wsrvm}, especially $f\in L^\infty (0,T; L_\gamma^1(\R^6))$
and $E,\, B \in L^\infty(0,T; L^2(\R^3))$, we obtain from the Lebesgue dominated convergence theorem,
\begin{multline}
  \label{Reste1}
  \int_0^T dt \int_{\R^3} dx  \left\{ \left(\int_{\R^3} d{\xi}\, \gamma f\right)  + \frac{|E|^2 + |B|^2}{2} \right\}\Lambda_R\partial_t\varphi
  \longrightarrow \\
  \int_0^T dt \int_{\R^3} dx  \left\{ \left(\int_{\R^3} d{\xi}\, \gamma f\right)  + \frac{|E|^2 + |B|^2}{2} \right\}\partial_t\varphi,
  \ \  {\rm as} \ \  R\rightarrow +\infty,
\end{multline}
and
\begin{equation}
  \label{Reste3}
  \int_0^T dt \int_{\R^3} dx
  \left\{ \left(\int_{\R^3} d{\xi}\, \gamma  f v \right)  + E\times B \right \}\cdot \nabla_x \Lambda_R \varphi
  \longrightarrow 0,\ \  {\rm as} \ \  R\rightarrow +\infty.
\end{equation}
Using \eqref{Reste1}-\eqref{Reste3}, and passing to the limit $R\rightarrow +\infty$ in \eqref{weakenergy2}, with $\Lambda(t,x)=\varphi(t)\lambda(x/R)$, 
we obtain,
\begin{equation}
  \label{weakenergy3}
  \int_0^T dt\, \partial_t\varphi \int_{\R^3} dx  \left\{ \left(\int_{\R^3} d{\xi}\, \gamma f\right)  + \frac{|E|^2 + |B|^2}{2} \right\}=0,
\end{equation}
which gives the global conservation law of total energy \eqref{TEC_G}.
\end{proof}

\begin{remark}
  \begin{itemize}
  \item[1.] If $E, \, B \in L^\infty(0,T;L^\infty(\R^6))$,
    we observe that the proof of Theorem~\ref{thm:EC1} remains valid
    without condition \eqref{thm:eq:MKEC1}, and then local and global
    conservation of total energy  \eqref{TEC_L}-\eqref{TEC_G} are satisfied.
  \item[2.] Using the continuous embedding $W^{\beta,q}(\R^3)\subset L^{3q/(3-\beta q)}(\R^3)$, with $\beta q <3$,
    we observe that if  $E, \, B \in L^\infty(0,T;L^2\cap W^{\beta,q}(\R^3))$, and
    \begin{equation*}
      \int_{\R^3} d{\xi}\, \gamma f \in L^\infty\big(0,T; L^{3q/((3+\beta)q-3)}(\R^3)\big), 
    \end{equation*}
    then estimates \eqref{eq:P3} and \eqref{eq:P4} still hold. Therefore
     local and global conservation of total energy  \eqref{TEC_L}-\eqref{TEC_G} are satisfied.
  \end{itemize}
\end{remark}

\section*{Acknowledgments}
The authors would like to thank Gregory Eyink for his constructive comments on this work.
The first and third authors wish to thank the Observatoire de la C\^ote d'Azur and the Laboratoire 
J.-L. Lagrange for their hospitality and financial support. TN's research was supported by the NSF under
grant DMS-1764119 and by an AMS Centennial Fellowship. Part of this work was done while TN was
visiting the Department of Mathematics and the Program in Applied and Computational Mathematics at Princeton University.

\end{document}